\documentclass[letterpaper]{amsart}

\title[The topological period--index problem over $6$-complexes]{The topological period--index
problem over $6$-complexes}

\author[Benjamin~Antieau]{Benjamin Antieau$^{*}$}
\address{University of Washington, Department of Mathematics, Box 354350, Seattle, WA 98195}
\email{benjamin.antieau@gmail.com}
\thanks{$^{*}$The author was supported in part by the NSF under Grant RTG DMS 0838697}

\author[Ben~Williams]{Ben Williams}
\address{University of British Columbia, Department of Mathematics, 1984 Mathematics Road
Vancouver, B.C., Canada V6T 1Z2}
\email{tbwillia@usc.edu}


\newcommand{\myauthor}{Benjamin Antieau and Ben Williams}
\newcommand{\mytitle}{The topological period-index problem over $6$-complexes}

\usepackage[pdfstartview=FitH,
            pdfauthor={\myauthor},
            pdftitle={\mytitle},
            colorlinks,
            linkcolor=reference,
            citecolor=citation,
            urlcolor=e-mail]{hyperref}

\usepackage{amsmath}
\usepackage{amscd}
\usepackage{amsbsy}
\usepackage{amssymb}
\usepackage{verbatim}
\usepackage{eufrak}
\usepackage{microtype}
\usepackage{hyperref}
\usepackage{mathrsfs}
\usepackage{amsthm}
\usepackage[all,cmtip]{xy}
\usepackage{tikz}
\usetikzlibrary{matrix,arrows}
\usepackage[abbrev,lite]{amsrefs}

\usepackage{mathpazo}

\usepackage{color}
\definecolor{todo}{rgb}{1,0,0}
\definecolor{conditional}{rgb}{0,1,0}
\definecolor{e-mail}{rgb}{0,.40,.80}
\definecolor{reference}{rgb}{.20,.60,.22}
\definecolor{mrnumber}{rgb}{.80,.40,0}
\definecolor{citation}{rgb}{0,.40,.80}

\DeclareMathAlphabet{\mathpzc}{OT1}{pzc}{m}{it}
\usepackage{dsfont}

\DeclareMathOperator{\Cl}{Cl}
\newcommand{\CCl}{\CC\mathrm{l}}
\newcommand{\BSO}{\mathrm{BSO}}

\newcommand{\Brt}{\Br_{\topo}}

\newcommand{\tot}{\mathrm{tot}}
\newcommand{\indt}{\ind_{\topo}}
\newcommand{\pert}{\per_{\topo}}

\newcommand{\PGL}{\mathrm{PGL}}
\newcommand{\Pc}{\mathrm{P}^c}
\newcommand{\Spin}{\mathrm{Spin}}
\newcommand{\PU}{\mathrm{PU}}
\newcommand{\SL}{\mathrm{SL}}
\newcommand{\GL}{\mathrm{GL}}

\newcommand{\tors}{{\mathrm{tors}}}

\DeclareMathOperator{\ord}{ord}

\DeclareMathOperator{\Spec}{Spec}

\DeclareMathOperator{\Pic}{Pic}

\DeclareMathOperator{\Hoh}{H}
\DeclareMathOperator{\Eoh}{E}

\DeclareMathOperator{\Br}{Br}

\DeclareMathOperator{\im}{im}

\DeclareMathOperator{\ind}{ind}

\DeclareMathOperator{\per}{per}
\DeclareMathOperator{\sk}{sk}

\DeclareMathOperator{\Pin}{Pin}

\newcommand{\topo}{{\mathrm{top}}}

\newcommand{\et}{\mathrm{\acute{e}t}}

\DeclareFontEncoding{OT2}{}{}

\newcommand{\K}{\mathbf{K}}

\DeclareMathOperator{\KU}{KU}

\newcommand{\G}{\mathbf{G}}

\DeclareMathOperator{\Mat}{Mat}

\renewcommand{\P}{\mathrm{P}}
\newcommand{\B}{\mathrm{B}}

\newcommand{\iso}{\cong}

\newcommand{\CC}{\mathds{C}}
\newcommand{\RR}{\mathds{R}}

\newcommand{\ZZ}{\mathds{Z}}

\newcommand{\PP}{\mathds{P}}

\newcommand{\Gm}{\mathds{G}_{m}}

\let\oldmarginpar\marginpar
\renewcommand\marginpar[1]{\-\oldmarginpar[\raggedleft\footnotesize #1]%
{\raggedright\footnotesize #1}}

\newcommand{\U}{\mathrm{U}}

\newcommand{\tensor}{\otimes}

\newcommand{\Brm}{\mathrm{B}}
\newcommand{\Erm}{\mathrm{E}}
\newcommand{\BPU}{{\Brm\PU}}
\newcommand{\BU}{\Brm \U}
\newcommand{\SU}{\mathrm{SU}}
\newcommand{\SO}{\mathrm{SO}}
\newcommand{\EU}{\mathrm{EU}}

\theoremstyle{plain}
\newtheorem{theorem}{Theorem}[section]
\newtheorem{lemma}[theorem]{Lemma}

\newtheorem{proposition}[theorem]{Proposition}

\newtheorem{conjecture}[theorem]{Conjecture}

\newtheoremstyle{named}{}{}{\itshape}{}{\bfseries}{.}{.5em}{#1 \thmnote{#3}}
\theoremstyle{named}
\newtheorem{namedtheorem}{Theorem}

\newtheoremstyle{strawman}{}{}{\itshape}{}{\bfseries}{.}{.5em}{#1}
\theoremstyle{strawman}
\newtheorem*{strawman}{Straw Man}

\theoremstyle{definition}
\newtheorem{definition}[theorem]{Definition}

\newtheorem{question}[theorem]{Question}

\theoremstyle{remark}

\begin{document}

\begin{abstract}
  By comparing the Postnikov towers of the classifying spaces of projective unitary groups and the differentials in a twisted
  Atiyah--Hirzebruch spectral sequence, we deduce a lower bound on the topological index in
  terms of the period, and solve the topological
  version of the period--index problem in full for finite CW complexes of dimension at most $6$.  Conditions are established that, if they
  were met in the cohomology of a smooth complex $3$-fold variety, would disprove the ordinary period--index conjecture. Examples of
  higher-dimensional varieties meeting these conditions are provided. We use our results to furnish an obstruction to realizing a period-$2$
  Brauer class as the class associated to a sheaf of Clifford algebras, and varieties are constructed for which the total Clifford invariant
  map is not surjective. No such examples were previously known.  
\end{abstract}

\maketitle

\section{Introduction and Summary of Results}

This paper is one of a sequence, \cites{aw1,aw2,aw4}, in which classical homotopy theory is used to derive statements about and derive intuition for the study of
division algebras, quadratic forms and the Brauer group of rings and schemes. We begin with an overview of the context of the paper, the
reader is referred to~\cite{aw1} for further details.

The cohomological Brauer group of a topological space $X$ is
$\Br'_{\topo}(X)=\Hoh^3(X,\ZZ)_{\tors}$. Given a principal bundle $P\rightarrow
X$ for $\PU_n$ there is a class $\delta_n(P) \in\Br'_{\topo}(X)$ associated to $P$ that is
the obstruction to lifting $P$ to a principal bundle for $\U_n$;
these classes arise in the exact sequence of cohomology associated to the extension \[1\rightarrow S^1\rightarrow\U_n\rightarrow\PU_n\rightarrow 1.\] For $\alpha \in\Br'_{\topo}(X)$,
the period of $\alpha$, written $\per_{\topo}(\alpha)$, is the order of $\alpha$ in the cohomological Brauer group, so that $\delta_n(P)$ is the first
measurement of the nontriviality of $P$. The index $\ind_{\topo}(\alpha)$ is an invariant of $\alpha$, being the greatest common divisor of all integers $n$ such that $\alpha=\delta_n(P)$ for some
$P$. If there is no such $P$, then $\indt(\alpha) = \infty$ by convention; we define the \textit{topological Brauer group} to be the
subset $\Brt(X) \subset \Brt'(X)$ consisting of classes $\alpha$ such that $\indt(\alpha) < \infty$. 

When $X$ is a finite CW complex the topological Brauer group is well-behaved. For instance, Serre
showed~\cite{grothendieck-brauer} that $\Br_{\topo}(X)$ and $\Br'_{\topo}(X)$ agree, and the authors were able to show that the period and
the index have the same prime divisors~\cite{aw1}*{Theorem 3.1}. In general, $\Brt(X)$ is a group, and $\pert(\alpha) | \indt(\alpha)$ for
$\alpha \in \Brt(X)$.

The definitions above for topological spaces were inspired by analogous ideas for schemes. In the case of a scheme the cohomological Brauer
group $\Br'(X)$ is $\Hoh^2_{\et}(X,\Gm)_{\tors}$. The classes $\alpha\in\Br'(X)$ appear as
obstructions to lifting principal bundles for $\PGL_n$ to
principal bundles for $\GL_n$ in the \'etale topology, and one can define $\per(\alpha)$ and $\ind(\alpha)$ in the same way as in topology. If $X$ is a
complex variety, the chief case considered in this paper, then there is a surjection $\Br'(X)\rightarrow\Br'_{\topo}(X)$. If
$\alpha\in\Br'(X)$, then we write $\bar{\alpha}$ for the image of $\alpha$ in $\Br'_{\topo}(X)$. One has
\begin{gather*}
    \per_{\topo}(\bar{\alpha})|\per(\alpha),\\
    \ind_{\topo}(\bar{\alpha})|\ind(\alpha)
\end{gather*}
in general.

The period--index problem is the problem of relating the period and index, especially by finding upper bounds on the index in terms of the period and the
dimension of the underlying object $X$. If $X$ is an irreducible scheme over a field $k$ we
write $k(X)$ for the function field of $X$. The following is one of the major outstanding conjectures in the study of division algebras.

\begin{conjecture}[Period--Index Conjecture for Function Fields]\label{conj:perind}
    Let $X$ be a $d$-dimensional irreducible variety over an algebraically closed field $k$.
    For any $\alpha\in\Br(k(X))$,
    \begin{equation*}
        \ind(\alpha)|\per(\alpha)^{d-1}.
    \end{equation*}
\end{conjecture}

Resolution of the conjecture would have important consequences for rational points of homogeneous spaces, moduli spaces of vector bundles,
quadratic forms, and division algebras.  Examples exist, see for instance~\cite{colliot}, where $\ind(\alpha) = \per(\alpha)^{d-1}$, so the
bound is sharp, if it holds. The conjecture has been proved in the case of $d=2$ by de Jong in~\cite{dejong}, but is not known for any
higher-dimensional varieties, not even for $\PP^3$.

A theorem of de Jong and Starr in~\cite{dejong-starr} implies that the period--index conjecture for $k(X)$ may be settled by considering only
classes lying in a subgroup called the unramified Brauer group; if $X$ is smooth and projective, then the Brauer group of $X$ coincides with
the unramified Brauer group of $k(X)$. This result recasts the period--index question, ostensibly a local and arithmetic question about
fields, as a geometric problem in the theory of smooth projective varieties.

In~\cite{aw1}, we formulated the period--index problem for topological spaces in the hope that the intuition gained in this setting might be
useful for the study of the cognate problem in algebraic geometry. We were able to deduce the existence, for any fixed $n$, of an integer $e(d,n)$ such that
\begin{equation*}
    \ind_{\topo}(\alpha)|\per_{\topo}(\alpha)^{e(d,n)}
\end{equation*}
for all classes $\alpha\in\Brt(X)$ whenever $X$ is a finite CW complex of dimension $2d$ and $\per_{\topo}(\alpha)=n$. An anologous result
for function fields and more generally $C_r$--fields was proved recently by Matzri~\cite{matzri}. When $n=2$, an earlier unpublished result
of Krashen established some similar bounds. We found that for $n$ having no prime divisors less than or equal to $2d+1$, one may take $e(d,n)=d$,
one greater than the exponent predicted by the algebraic period--index conjecture under the analogy between $d$-dimensional smooth
projective algebraic varieties and $2d$-dimensional CW complexes.

The result of de Jong and Starr drives us to investigate the indices of classes in the Brauer groups of
projective $d$-folds over the complex numbers.  Topologically these $d$-folds have the homotopy type of real $2d$-dimensional closed manifolds, and
\textit{a fortiori} of finite CW complexes. 
Our investigations began with this in mind, as well as the following na\"ive hypothesis:

\begin{strawman} If $X$ is a $2d$-dimensional finite CW complex, and $\alpha\in\Brt(X)$, then
    \begin{equation*}
        \ind_{\topo}(\alpha)|\per_{\topo}(\alpha)^{d-1}.
    \end{equation*}
\end{strawman}

In order to state the main computational theorem, we first must set up notation for the unreduced Bockstein map 
\[\beta_n : \Hoh^i(X, \ZZ/n) \to \Hoh^{i+1}(X, \ZZ)\]
and for a cohomology operation, a \textit{Pontryagin square}, 
\[ P_2: \Hoh^2( X, \ZZ/(2m)) \to \Hoh^4(X, \ZZ/(4m)) \]
with the property that $2P_2(\xi)$ is the image of $\xi^2$ in $\Hoh^2(X, \ZZ/(4m))$.

Our results depend only on the homotopy types of the CW complexes involved, so that a result asserted to hold for a CW complex of dimension
$d$ will hold for all homotopy-equivalent CW complexes, possibly of higher or infinite dimension.

\begin{namedtheorem}[A]
  Let $X$ be a connected finite CW complex, let $\alpha \in \Hoh^3(X,\ZZ)_\tors = \Brt(X)$
  be a Brauer class, and write $n = \pert(\alpha)$.
  Choose $\xi \in \Hoh^2(X,\ZZ/n)$ such that $\beta_n(\xi) = \alpha$, let $Q(\xi) $ denote either $\beta_n(\xi^2)$ if $n$ is odd or
  $\beta_{2n}(P_2(\xi))$ if $n$ is even, and let $\tilde Q(\xi)$ denote the reduction of $Q(\xi)$ to $\Hoh^5(X, \ZZ)/ ( \alpha \smile
  \Hoh^2(X, \ZZ))$. Then:
  \begin{enumerate}
  \item The class $\tilde Q (\xi)$ depends only on $\alpha$. We write $\tilde Q(\alpha)$ for
      this class.
  \item The order of $\tilde Q(\xi)$ divides $\pert(\alpha)$ if $\pert(\alpha)$ is odd and divides $2 \pert(\alpha)$ if $\pert(\alpha)$
    is even.
  \item For any $\alpha$ and $\xi$,
    \begin{equation*}
      \pert(\alpha) \ord(\tilde Q(\xi)) | \indt(\alpha),
    \end{equation*}
    with equality if the dimension of $X$ is at most $6$.
  \end{enumerate}    
\end{namedtheorem}

Theorem A is proved by studying a sort of generalized cohomology operation, denoted by $G$.
We will define $G(\alpha)$ via a differential in the twisted Atiyah-Hirzebruch spectral
sequence associated to $\alpha$. Then, we prove that $G(\alpha)$ and $\tilde{Q}(\alpha)$
must generate the same subgroup of $\Hoh^5(X,\ZZ)/(\alpha\smile\Hoh^2(X,\ZZ))$.

The first consequence of the theorem is a disproof of the straw-man conjecture.

\begin{namedtheorem}[B]
    Let $n$ be a positive integer, and let $\epsilon(n)$ denote $n\gcd(2,n)$. There exists a connected finite CW complex $X$ of
    dimension $6$ equipped with a class $\alpha \in \Brt(X)$ for which $\pert(\alpha) = n$ and $\indt(\alpha)=\epsilon(n)n$.
\end{namedtheorem}

Thus, if $X$ is a finite CW complex of dimension at
most $6$ and $\alpha\in\Brt(X)$, then
\begin{equation*}
    \indt(\alpha)\left|\begin{cases}
        \pert(\alpha)^2     &   \text{if $\pert(\alpha)$ is odd,}\\
        2\pert(\alpha)^2    &   \text{if $\pert(\alpha)$ is even.}
    \end{cases}\right.
\end{equation*}
These bounds are moreover sharp. Together with the facts that $\pert(\alpha)$ and
$\indt(\alpha)$ are positive and have the same prime divisors, the bounds are the only restriction on the pair of integers
$(\pert(\alpha),\indt(\alpha))$ for Brauer classes on such a space $X$.

The theory may also be applied to the period--index conjecture for function fields.

\begin{namedtheorem}[C]
  If there exists a smooth projective complex $3$-fold $X$ and a $2$-torsion class
  $\alpha\in\Br(X)$ such that $2\tilde{Q}(\bar{\alpha})\neq
  0$, where $\bar{\alpha}$ is the image of $\alpha$ in $\Hoh^3(X,\ZZ)_{\tors}$, then the period--index conjecture fails at the prime
  $2$ for the function field $\CC(X)$ of $X$: the image of the class $\alpha$ in $\Br(\CC(X))$ has period $2$ and index at least $8$.
\end{namedtheorem}

At present, we are unable to furnish any examples of smooth projective complex $3$-folds meeting the conditions of the theorem.

The theorem is proved in Section~\ref{sec:3folds} via two comparison results, the first comparing the topological and algebraic
period--index problems on $X$, the second comparing the period--index problem on the scheme $X$ with the problem for the field $\CC(X)$ by
way of a result suggested to the first author by David Saltman at the AIM workshop ``Quadratic forms, division algebras, and patching'' in
January 2011.

We are able to furnish examples of smooth projective or smooth affine schemes and classes $\bar{ \alpha} \in \Brt(X)$ for which $2
\tilde{Q}(\alpha) \neq 0$, so Theorem C presents a dichotomy: either there are previously unknown
restrictions on the $2$-torsion in the singular cohomology of smooth projective $3$-folds that are peculiar to the $3$-dimensional case, or
the period--index conjecture is false. For later use we remark that there exists an affine $5$-fold complex variety $X$ having the homotopy type of a
$5$-dimensional CW complex and a class $\alpha \in \Br(X)$ such that $\pert(\bar \alpha) =
2$ and $\indt(\bar \alpha) = 8$.

We may further apply Theorem C to the study of quadratic forms on schemes. There is a map, the Clifford map, which assigns a
central simple algebra to an even-dimensional quadratic module of trivial discriminant, $(V,q)$, compatibly with direct summation of
quadratic modules; quaternion algebras are a particular case of this construction. The Clifford map yields a Clifford invariant map $\Cl:
I^2(k) \to {_2\Br(k)}$ from the group of Witt classes of even rank and trivial discriminant to the $2$--torsion in the Brauer group. A
theorem of Merkurjev, \cite{merkurjev}, states that when the characteristic of $k$ is not $2$, the Clifford invariant map is
surjective.

As the notions of quadratic module, Witt group, Brauer group, and Clifford algebra extend to commutative rings and schemes in general, Alex
Hahn asked whether there existed a counterexample to the theorem of Merkurjev in the context of rings. A counterexample was first constructed
by Parimala and Sridharan in \cite{parimala-sridharan-1} over a smooth proper curve over a local field, although they observed that for
affine curves over local fields the map is surjective. In~\cite{parimala-sridharan-2}, they subsequently employed Jouanolou's device to construct an example
of an affine variety over a local field where the Clifford invariant map is not surjective, thus answering Hahn's question.

For fifteen years the matter seemed settled, until Auel showed in~\cite{auel} that the examples of Parimala and Sridharan were explained by an
obstruction in the Picard group. Over a scheme $X$, there is a total Witt group $W_{\tot}(X)$ constructed using line-bundle-valued quadratic
modules and a subgroup $I^2_{\tot}(X)$ of even-dimensional forms with trivial discriminant. The total Clifford invariant map
$\Cl_\tot:I^2_{\tot}(X)\rightarrow{_2\Br(X)}$ is a generalization of the simple Clifford invariant map to befit the case of nontrivial
Picard group.

\begin{question}\label{quest:surjective}
    Is the total Clifford invariant map $\Cl_{\tot}:I^2_{\tot}(X)\rightarrow{_2\Br(X)}$ surjective?
\end{question}

This question is posed by Auel at the end of~\cite{auel} for schemes $X$ for which the function field has $2$-cohomological dimension at
most $3$.  We show that Question~\ref{quest:surjective} has a negative answer by constructing a smooth affine complex $5$-fold $X$ where
$\Cl_{\tot}$ is not surjective. The example we construct is such that $\Pic(X)=0$ so the problem reduces to the surjectivity of $\Cl$.

\begin{namedtheorem}[D] \label{thm:d}
    There exists a $5$-dimensional smooth affine complex variety $X$ such that
    $\Br(X)=\ZZ/2$, but the total Clifford invariant map vanishes.
\end{namedtheorem}

The proof shows that the topological index of a class in the image of the total Clifford invariant map must divide $4$ when $X$ has the
homotopy type of a $6$-dimensional CW complex. The affine $5$-fold $Y$, previously encountered in our discussion of Theorem C, now
serves to provide a counterexample. We remark that a smooth projective $3$--fold for which the period--index conjecture were to fail for the reasons given in
that theorem could be modified to furnish a negative answer to Auel's low-dimensional question.





\textbf{Acknowledgments.} The authors would like to thank Asher Auel for conversations about
the total Clifford invariant, Johan de Jong for pointing out a mistake in an argument in an
earlier version, and David Saltman for explaining that the index of an unramified call can
be computed either globally or over the function field. We would like to thank a referee for
a thorough reading of the paper, which allowed us to fix many mistakes and improve the exposition.

\section{The Existence of $G$}\label{sec:g}

This and the following three sections of this paper are devoted to the definition and calculation of a kind of generalized cohomology
operation, denoted $G$. The intricacy in the definition of $G$ is that we should like to say that for a connected CW complex $X$, the
operation $G_X$ is defined on $\Hoh^3(X, \ZZ)_\tors$, and that $G_X(\alpha)$ is an element of the group $\Hoh^5(X, \ZZ)/ (\Hoh^2(X, \ZZ)
\smile \alpha)$, which is a group that varies with $\alpha$.

We shall therefore construct an operation $G$ which is an assignment to any connected CW complex $X$ of a function $G_X : \Hoh^3(X, \ZZ)_\tors
\to \mathcal{P}(\Hoh^5(X, \ZZ))$, where $\mathcal{P}(S) $ denotes the power set of $S$, such that $G_X(\alpha)$ is a coset of the subgroup
$\Hoh^2(X, \ZZ) \smile \alpha$ in $\Hoh^5(X, \ZZ)$. We shall generally write $G(\alpha)$ for $G_X(\alpha)$, since the class $\alpha$
determines $X$, and we shall understand $G(\alpha)$ as an element of the group $\Hoh^5(X, \ZZ) / (\Hoh^2(X, \ZZ) \smile \alpha)$. The
construction of $G$ is as the image of a certain element under a $d_5$-differential in a spectral sequence.

The coefficient sheaf $\ZZ(q/2)$ is defined to be $\ZZ$ if $q$ is even and $0$ if $q$ is odd. We remind the reader that all results applying
to CW complexes depend only on the homotopy type of the complexes involved, but for ease of exposition we write, for example, the brief ``a
CW complex $X$ of dimension $6$'' rather than the longer ``a space $X$ having the homotopy type of a CW complex of dimension $6$''.

\begin{proposition}\label{p:descSS}
    Let $X$ be a topological space, and let $\alpha\in\Hoh^3(X,\ZZ)_{\tors}$.
    There is an $\alpha$-twisted Atiyah--Hirzebruch spectral sequence, functorial in the pair $(X, \alpha)$, 
    \begin{equation*}
        \Eoh_2^{p,q}=\Hoh^p(X,\ZZ(q/2))\Longrightarrow \KU^{p+q}(X)_{\alpha}
    \end{equation*}
    with differentials $d_r^{\alpha}$ of degree $(r,-r+1)$, which converges strongly if $X$ is a finite-dimensional CW complex.
    The edge map $\KU^0(X)_{\alpha}\rightarrow\Hoh^0(X,\ZZ)$ is the rank map.
\end{proposition}
\begin{proof}
        This is~\cite[Theorem~4.1]{atiyah-segal}.
\end{proof}

The spectral sequence is functorial in that a map $f: Y \to X$ and a class $\alpha \in \Hoh^3(X, \ZZ)$ give
rise to a map of spectral sequences compatible with the map $\KU^*(X)_\alpha \to \KU^*(Y)_{f^*(\alpha)}$. 

\begin{proposition} \label{p:d31}
  Let $X$ be a connected CW complex. The differential
    \begin{equation*}
      d_3^{\alpha}:\Hoh^0(X,\ZZ)\longrightarrow\Hoh^3(X,\ZZ)
    \end{equation*}
    sends $1$ to $\alpha$.
\end{proposition}
 \begin{proof}
        See~\cite{atiyah-segal-cohomology}*{Proposition 4.6} or~\cite{aw1}*{Proposition~2.4}.
    \end{proof}

The twisted spectral sequence is a module over the untwisted Atiyah-Hirzebruch spectral sequence By writing a class $x\in\Hoh^p(X,\ZZ(q/2))$
in the twisted spectral sequence as $x\cdot 1$, where $1$ is an element in the twisted spectral sequence and $x$ in the untwisted, we see
that $d_3^{\alpha}(x)=x\smile d_3^{\alpha}(1)+d_3(x)\smile 1=x\smile\alpha+d_3(x)$. The untwisted differential $d_3$ vanishes on
$\Hoh^2(X,\ZZ)$ since there are no non-zero cohomology operations $K(\ZZ,2)\rightarrow K(\ZZ,5)$ and therefore
$\Eoh_5^{5,-4}=\Hoh^5(X,\ZZ)/(\Hoh^2(X,\ZZ)\smile\alpha)$.

\begin{proposition} \label{p:indKasPermCyc}
  Let $X$ be a connected finite CW complex and let $\alpha \in \Hoh^3(X, \ZZ)_{\tors}$. Then
  $\ind_{\topo}(\alpha)$ is the positive generator of
  the subgroup $\Eoh_\infty^{0,0}$ of permanent cycles in $\Eoh_2^{0,0} = \Hoh^0(X, \ZZ) = \ZZ$ in the Atiyah-Hirzebruch spectral sequence of Proposition \ref{p:descSS}.
\end{proposition}
\begin{proof}
    This is an amalgamation of Proposition 2.21 and Lemma 2.23 of \cite{aw1}.
  \end{proof}

Given a connected CW complex $X$ and a class $\alpha \in \Br'(X)$ we define $G_X(\alpha)$ as follows. 

\begin{definition} \label{d:wasA}
    Consider the $\alpha$-twisted Atiyah--Hirzebruch spectral sequence. The kernel of the first nonzero differential, $d^\alpha_3$, emanating
    from $\Eoh^{0,0}_*$ is $\pert(\alpha) \ZZ$, by Proposition \ref{p:d31}. The $d^\alpha_4$
    differential vanishes for degree reasons. The
    $d_5^{\alpha}$-differential emanating from $\Eoh_5^{0,0}$ takes the form $d^\alpha_5 : \pert(\alpha) \ZZ \to \Hoh^5(X,\ZZ)/(\Hoh^2(X,
    \ZZ)\smile\alpha)$. We define $G_X(\alpha)$ to be the element $d_5^\alpha( \pert(\alpha))$ in $\Hoh^5(X, \ZZ)/ (\Hoh^2(X, \ZZ) \smile \alpha)$.
\end{definition}

We write $G(\alpha)$ for $G_X(\alpha)$ in the sequel.  

\begin{theorem} \label{thm:wasA}
  If $f: Y \to X$ is a map of connected CW complexes and if $\alpha \in \Br'(X)$, then $G$ satisfies the naturality condition
  \begin{equation} \label{eq:nat} f^* G(\alpha) = \frac{\pert(\alpha)}{\pert(f^*(\alpha))} G(f^* \alpha). \end{equation}
  For any connected finite CW complex, $X$ and class $\alpha \in \Br'(X)$, the relation
  \begin{equation} \label{eq:div1} \pert(\alpha) \ord(G(\alpha)) | \indt(\alpha) \end{equation}
  holds, and, if $X$ is of dimension no greater than $6$, then
  \begin{equation} \label{eq:div2} \pert(\alpha) \ord(G(\alpha)) = \indt(\alpha). \end{equation}
\end{theorem}
\begin{proof}
  Relation \eqref{eq:nat} follows immediately from the functoriality of the Atiyah--Hirzebruch spectral sequence.
  If $X$ is a finite CW complex, $\indt(\alpha)$ is a generator of the subgroup of permanent cocycles in $\Eoh^{0,0}_\infty$ by Proposition
  \ref{p:indKasPermCyc}. The subgroup generated by $\pert(\alpha)\ord(G(\alpha))$ in $\Eoh_2^{0,0}$ is the subgroup $\Eoh_6^{0,0}$, from
  which \eqref{eq:div1} follows. If $X$ has dimension no greater than $6$,
  there are no further nonzero differentials emanating from $\Eoh_i^{0,0}$, and so $\Eoh_6^{0,0} = \Eoh_{\infty}^{0,0}$, which establishes \eqref{eq:div2}.
\end{proof}

\section{The Low-Degree Cohomology of $\BPU_n$}\label{sec:cohbpu}

The previous section presented the construction of an operation, $G$. The next three sections set forth a procedure for determining the subgroup
of $\Hoh^5(X, \ZZ)/ ( \Hoh^2(X, \ZZ) \smile \alpha)$ generated by $G(\alpha)$ and in particular for calculating $\ord(G(\alpha))$. 

The method originates in the obstruction-theory of $\BPU_n$. For a finite CW complex $X$ and for any $\alpha \in \Br'(X)$ there exists some
$n$ such that the class $\alpha$ is in the image of a map $\Hoh^3(\BPU_n, \ZZ) \to \Hoh^3(X, \ZZ)$. As a consequence, determining the
subgroup generated by $G$ on the spaces $\BPU_n$ collectively will suffice to determine the subgroup generated by $G$ on all connected CW
complexes.

To avoid trivialities, we shall assume throughout that $n$ is an integer greater than $1$. In this section and henceforth, the notation $\epsilon(n)$
will be used to denote $n\gcd(2,n)$.

This section in particular is devoted to establishing the cohomology $\Hoh^*(\BPU_n, \ZZ)$ in low degrees. We show that
\begin{align*}
  \Hoh^i(\BPU_n, \ZZ) & = 0, \quad i \le 2  & \Hoh^3(\BPU_n, \ZZ) & = \ZZ/n \\
 \Hoh^4( \BPU_n, \ZZ)  & = \ZZ & \Hoh^5(\BPU_n, \ZZ) & = 0.
\end{align*}
Among these, the results for $\Hoh^i(\BPU_n, \ZZ)$ where $i \le 3$ are trivial, given the calculation of $\Hoh^*(\PU_n, \ZZ)$ of
\cite{baum-browder}. The fact $\Hoh^4(\BPU_n, \ZZ) = \ZZ$ will appear in passing, but it is showing $\Hoh^5(\BPU_n, \ZZ)= 0$ that will be of
greatest significance, and this is what occupies our attention.

\begin{proposition}\label{prop:cohbpun}
   $\Hoh^5(\BPU_n, \ZZ) = 0$.
\end{proposition}

This statement can be extracted from either~\cite{kameko-yagita} or~\cite{vistoli} for $\BPU_p$ when $p$ is an odd
prime. The special case of $\BPU_2$ follows from~\cite{brown} and the exceptional
isomorphism $\PU_2\iso\SO_3$. We give a proof which works for all $n$.

To prove the proposition, we need several lemmas.  Recall~\cite{milnor-stasheff} that the ring $\Hoh^*(\BU_n,\ZZ)$ is a polynomial algebra
in Chern classes $c_1, c_2, \dots, c_n$, with $|c_i| = 2i$; recall also that $\Hoh^*(\U_n, \ZZ)$ is the exterior algebra $\Lambda_\ZZ(y_1,
\dots, y_n)$ with $|y_i| = 2i-1$.  Write $\rho: \U_n \to \PU_n$ for the quotient map and $\Brm\rho:\BU_n\rightarrow\BPU_n$ for the
associated map on classifying spaces.

\begin{lemma}[Woodward~\cite{woodward}\footnote{The reader is warned that there is a slight mistake in Woodward's paper, explained and corrected in
    our note~\cite{aw5}. The error does not impinge on the present work.}]
    \label{l:woodward} If $n\geq 2$ is odd, then $\Brm\rho^*$ identifies
  $\Hoh^4(\BPU_n, \ZZ)$ with the subgroup of $\Hoh^4(\BU_n,\ZZ)$ generated by $\frac{n-1}{2}c_1^2 -n c_2$. If $n$ is even, then $\Brm\rho^*$
  identifies $\Hoh^4(\BPU_n,\ZZ)$ with the subgroup of $\Hoh^4(\BU_n,\ZZ)$ generated by $(n-1)c_1^2 - 2n c_2$.
\end{lemma}

The proof is included for the sake of completeness.

\begin{proof}
  There is a fiber sequence $\Brm S^1\to \BU_n \to \BPU_n$, where the map $\Brm\Delta: \Brm S^1 \to \BU_n$ is that induced by the
  diagonal. Since $\tilde \Hoh^i(\BPU_n, \ZZ)=0$ for $i \le 2$, and since $\Hoh^*(\Brm S^1, \ZZ)= \ZZ[t]$ where $|t|=2$, the exact
  sequence of low-dimensional terms in the Serre spectral sequence for $\Brm S^1\rightarrow\BU_n\rightarrow\BPU_n$ contains
    \begin{equation*}
      0 =\Hoh^3(\Brm S^1 , \ZZ) \to \Hoh^4(\BPU_n, \ZZ) \overset{\Brm \rho^*}{\longrightarrow}
      \Hoh^4(\BU_n, \ZZ) \overset{\Brm \Delta^*}{\longrightarrow} \Hoh^4(\Brm S^1, \ZZ).
    \end{equation*}
    The map $\Brm \rho^*$ is the inclusion $\ker \Brm \Delta^* \to \Hoh^4(\BU_n, \ZZ)$.

    The pull-back along $\Delta$ of the universal bundle $\EU_n \to \BU_n$ has total Chern class $(1+t)^n$, from which it follows that $\Brm
    \Delta^*(c_1) = nt$, implying that $\Brm \Delta^*(c_1^2) = n^2t^2$ and $\Brm\Delta^*(c_2) = \binom{n}{2} t^2$. The description of
    $\ker\Brm\Delta^* = \Hoh^4(\BPU_n, \ZZ)$ is now elementary.
\end{proof}

Write $d$ for the unique class in $\Hoh^4(\BPU_n,\ZZ)$ such that
$\Brm\rho^*(d)=\frac{n-1}{2}c_1^2-nc_2$ when $n$ is odd, or $\Brm\rho^*(d)=(n-1)c_1^2-2nc_2$
when $n$ is even.

\begin{lemma}\label{l:cohSusUn}
  In the cohomology Serre spectral sequence for $\U_n\rightarrow\Erm\U_n\rightarrow\BU_n$  one has
    \begin{align*}
      &  d_2(y_1)    =c_1, && d_3(y_3)=0,\\
      &  d_2(y_3)    =0, &&    d_4(y_3)=c_2\,(\mathrm{mod}\, {c_1^2}).
    \end{align*}
    up to a change of generators
    \begin{proof}
          This follows immediately from the contractibility of $\Erm \U_n$.
    \end{proof}
\end{lemma}

\begin{lemma} \label{l:cohoPUn}
    The low-degree cohomology of $\Hoh^*(\PU_n,\ZZ)$ is
    \begin{align*}
       & \Hoh^1(\PU_n,\ZZ)   =0, &&  \Hoh^3(\PU_n,\ZZ)   =\ZZ\cdot y,\\
       &  \Hoh^2(\PU_n,\ZZ)   =\ZZ/n\cdot x, &&  \Hoh^4(\PU_n,\ZZ)   =\ZZ/m\cdot x^2, 
   \end{align*}
    where $m=n$ if $n$ is odd and $m=n/2$ if $n$ is even. The class $y$ is uniquely
    determined by the requirement that $\rho^*(y)=y_3$ when $n$ is odd, $\rho^*(y)=2y_3$
    when $n$ is even.
    \begin{proof}
        See~\cite{baum-browder}*{Section 4}.
    \end{proof}
\end{lemma}


\begin{proof}[Proof of Proposition]
  The integer $m$ is defined as in Lemma~\ref{l:cohoPUn}.

    The lower-left corner of the Serre spectral sequence associated with $\PU_n \to \Erm \PU_n \to \BPU_n$ is displayed in Figure \ref{fig:ssspun}.
    \begin{figure}[h]
        \centering
        \begin{equation*}    \xymatrix@R=8px@C=8px{ (\ZZ/m)\cdot x^2 \ar@{^(->}^{d_3}[ddrrr] \\
            \ZZ\cdot y \ar@{->>}^{d_2}[drr] \ar@/_10px/^{d_4}[dddrrrr] \\ (\ZZ/n)\cdot x
            \ar@/_10px/^{\iso}_{d_3}[ddrrr] & & \ZZ/n & \ZZ/n \\ 0 \\ \ZZ & 0 & 0  & \Hoh^3(\BPU_n, \ZZ) &
              \Hoh^4(\BPU_n, \ZZ) & \Hoh^5(\BPU_n,\ZZ) }
        \end{equation*}
        \caption{The $\Eoh_3$-page of the Serre spectral sequence associated with $\PU_n \to \Erm \PU_n \to \BPU_n$.}
        \label{fig:ssspun}
    \end{figure}
    Since the spectral sequence converges to the cohomology of the contractible space $\Erm\PU_n$ we can deduce a great deal about the terms and
    differentials. In the first place, the $d_3$-differential $d_3 :(\ZZ/n)\cdot x \to
    \Hoh^3(\BPU_n, \ZZ)$ is an isomorphism.
    
    In the second place we show that the map labeled $d_2$ in Figure \ref{fig:ssspun} is surjective. We refer to this map simply as $d_2$. We deduce the existence of a commutative square
   \begin{equation}
      \label{eq:2}
      \xymatrix{ \ker (d_2) \ar^{d_4}@{->>}[d] \ar@{^(->}[r] & \ZZ\cdot y \ar^{\rho^*}[r] & \ZZ\cdot y_3 \ar^{d_4}[d] \\
        \Hoh^4 ( \BPU_n, \ZZ) \ar[rr]& & \dfrac{\Hoh^4(\B\U_n, \ZZ)}{ \ZZ \cdot c_1^2} \iso \ZZ\cdot c_2}
    \end{equation}
    where the vertical map on the right is that of the Serre spectral sequence associated to the fiber sequence $\U_n \to \Erm \U_n \to
    \BU_n$. The subgroup $\ker(d_2)$ is generated by an element $by$ such that that $d_4(by)$ is a generator of $\Hoh^4(\BPU_n,\ZZ)$;
    following the arrows counterclockwise and using Lemma \ref{l:woodward}, we see that $by$ has image $\pm \epsilon(n)c_2$ in $\ZZ\cdot
    c_2$. Following the arrows clockwise, and using Lemmas \ref{l:cohSusUn} and \ref{l:cohoPUn}, we see that $by$ has image $\pm bc_2$ in $\ZZ
    \cdot c_2$ when $n$ is odd and image $ \pm 2bc_2$ when $n$ is even. In particular, it follows that $b= \pm n$. Since $\ker(d_2) = n\ZZ\cdot y$,
    we deduce that $d_2$ is surjective.

    In the third place, $d_3(x^2) = 2 x d_3(x)$, so that $a x^2$ is in the kernel of the $d_3$-differential if and only if $2a \equiv 0
    \!\!\pmod{n}$, which is to say that $a \equiv 0 \!\!\pmod{m}$ and so the $d_3$-differential is injective on $\ZZ/m \cdot x^2$. One may now check
    directly that there are no nonzero groups which can support a differential $d_i : A \to \Hoh^5(\BPU_n, \ZZ)$, so the group
    $\Hoh^5(\BPU_n,\ZZ)$ consists of permanent cocycles, and since $\Erm\PU_n$ is contractible it that follows that $\Hoh^5(\BPU_n, \ZZ) = 0$.
\end{proof}

\section{The Obstruction Theory of $\BPU_n$}\label{sec:obstrbpu}

This section takes up the work of the previous one.

By definition, the topological index $\ind_\topo(\alpha)$ is the greatest common divisor of the integers $n$ such that $\alpha
\in \Hoh^3(X, \ZZ) = \Hoh^2(X , S^1)$ is in the image of the connecting homomorphism $\Hoh^1( X, \PU_n) \to \Hoh^2(X, S^1)$. In homotopical
terms, it is the greatest common divisor of the integers $n$ such that a lift exists in the diagram
\begin{equation}
  \label{eq:5}
  \xymatrix{ & \BPU_n \ar[d] \\ X \ar^\alpha[r] \ar@{-->}[ur] & K(\ZZ, 3). }
\end{equation}
Here the map $\BPU_n \to K(\ZZ, 3)$ is the application of the classifying-space functor to the map $\PU_n \to \Brm S^1$ which classifies
the quotient map $\U_n \to \PU_n$. In particular, the obstruction theory of the spaces $\BPU_n$ as $n$ varies should inform calculations of the topological index.

The low-degree part of a Postnikov tower of $\BPU_n$ takes the form
\[ \xymatrix{ K(\ZZ, 4) \ar[r] & \BPU_n[4] \ar[d] \\ & K(\ZZ/n, 2) \ar@{..>}^{f_n}[r] & K(\ZZ, 5) } \] where the dotted arrow, $f_n$, is the
classifying map of the fiber sequence represented by the other two arrows. Our main result here is the determination of $f_n$. Since $f_n:
K(\ZZ/n, 2) \to K(\ZZ, 5)$ is a class in the cohomology of an Eilenberg--MacLane space, it is a cohomology operation, and we may hope to
describe it in simple terms. In fact, we shall show that if $n$ is odd, then $f_n$ agrees, up to multiplication by a unit, with
$\beta_n(\iota^2)$, the Bockstein of the square of the tautological class. If $n$ is even, $f_n$ agrees, up to multiplication by a unit,
with $\beta_{2n}(P_2(\iota))$, where $P_2$ is an operation such that $2 P_2(\iota)$ agrees with the image of $\iota^2$ in
$\Hoh^4(K(\ZZ/n,2), \ZZ/2n)$.

In this section and hereafter we tacitly fix basepoints for all connected topological spaces appearing, and we shall write $\pi_i(X)$ in
lieu of $\pi_i(X, x)$. Once again, we assume that $n$ is an integer satisfying $n\ge 2$. If $A$ is an abelian group and $g \in A$ an
element, we shall write $\langle g \rangle$ for the cyclic subgroup of $A$ generated by $g$.

\subsection*{The homotopy groups of $\BPU_n$.}

Since the task of this section is to determine the low-dimensional obstruction theory of $\BPU_n$ for varying $n$, it will be necessary to describe the
homotopy groups of $\BPU_n$ and the maps between these groups induced by certain standard maps.

For future use, we define
\[ \Pc(n,rn) = \SU_{rn} / (\ZZ/n) \] 
where $\ZZ/n$ is viewed as a subgroup of $\ZZ/(rn)$, the center of $\SU_{rn}$, in the evident way. Note that $\Pc(n,n) = \PU_n$.
From the long exact sequence of a fibration, we deduce
\begin{equation*}
    \pi_i(\Pc(n,rn))\iso\begin{cases}
        \ZZ/n   &   \textrm{if $i=1$,}\\
        \pi_i(\SU_{rn}) &   \textrm{otherwise.}
    \end{cases}
  \end{equation*}
The following homotopy calculations follow directly from Bott periodicity,~\cite{bott}*{Theorem 5}.
\begin{equation*}
    \pi_i(\SU_n)\iso\begin{cases}
        0   &   \textrm{if $i<2n$ is even or $i=1$,}\\
        \ZZ &   \textrm{if $i<2n$ is odd and $i\neq 1$,}\\
        \ZZ/n! &  \textrm{if $i=2n$.}
    \end{cases}
\end{equation*}
Moreover, the standard inclusion maps $\SU_n\rightarrow\SU_{n+1}$ induce isomorphisms
\begin{equation*}
    \pi_i(\SU_n)\rightarrow \pi_i(\SU_{n+1})
\end{equation*}
whenever $i<2n$. The first two stages of the Postnikov tower of $\B\Pc(n,rn)$ take the form
\begin{equation}
  \label{eq:6}
  \xymatrix{ K(\ZZ, 4) \ar[r] & \B\Pc(n,rn)[4] \ar[d] \\ & \B\Pc(n,rn)[2] \simeq K(\ZZ/n , 2).}
\end{equation}
This is a principal $K(\ZZ, 4)$--bundle, classified by a map $K(\ZZ/n, 2) \to K(\ZZ, 5)$.

Two families of maps between the $\Pc(n,rn)$ are easily defined. In the first place, one may take a further quotient: $\Pc(n, \ell rn) \to
\Pc(rn, \ell rn)$. In the second, one may take the direct sum of $\ell$ copies of a matrix in $\SU_{rn}$, giving a map $\oplus^\ell:
\SU_{rn} \to \SU_{\ell rn}$. This map is compatible with the action of $\ZZ/n$ and therefore descends to a map $\oplus^\ell: \Pc(n,
rn) \to \Pc(n ,\ell rn)$.

\begin{lemma}\label{lem:homotopymaps}
  The map induced by reduction $\pi_2(\B \Pc(n, \ell rn)) \to \pi_2(\B \Pc( rn , \ell rn))$ is an inclusion $\ZZ/n \to \ZZ/(rn)$. On higher
  homotopy groups this map induces an isomorphism.
  
  The map $\pi_i(\B\Pc(n,rn)) \to \pi_i(\B\Pc(n, \ell rn))$ induced by $\ell$-fold direct summation
  \begin{enumerate}
  \item is an isomorphism $\pi_2(\B \Pc(n,rn)) \to \pi_2(\B\P(n,\ell rn))$ when $i=2$,
  \item is the composite of the stabilization map $\pi_i(\B \SU_{rn}) \to \pi_i(\B \SU_{\ell rn})$ and multiplication by $\ell$ when
    $i>2$.
  \end{enumerate}
\end{lemma}
\begin{proof}
  The first statement is follows from the diagram of fiber sequences
  \begin{equation*}
    \xymatrix{ \ZZ/n \ar[r] \ar[d] & \SU_{\ell rn} \ar[r] \ar@{=}[d] & \Pc(n, \ell rn) \ar[d] \\ \ZZ/(rn) \ar[r] & \SU_{\ell rn} \ar[r] &
      \Pc(rn , \ell rn). }
  \end{equation*}

  The map induced by $\ell$-fold direct summation $\SU_{rn} \to \SU_{\ell rn}$ on homotopy groups is readily seen to be the composite of the map
  $\pi_i(\SU_{rn}) \to \pi_i(\SU_{\ell rn})$ induced by the standard inclusion $\SU_{rn} \to
  \SU_{\ell rn}$ with multiplication by $l$ on $\pi_i(\SU_{\ell rn})$,
  by an Eckmann--Hilton argument. The result for $\Pc(n,\ell rn)$ follows by comparison.
\end{proof}

We attend mainly to $\BPU_n = \B\Pc(n,n)$ in this section; some other spaces $\B \Pc(n,rn)$ appear later.

We shall use the notation $X[n]$ for the $(n+1)$-st coskeleton of a simple, connected CW complex, which is to say that
there is a natural map $X \to X[n]$ which induces an isomorphism on homotopy groups $\pi_i(\cdot)$ when $i\le n$, and such that $\pi_i(X[n])
= 0$ when $i > n$. Since the spaces $X$ are always simple in our treatment, the space $X[n]$ may be taken to be the total space of a
fibration appearing in the total space of a Postnikov tower converging to $X$.

\subsection*{The cohomology of $K(\ZZ/n, 2)$}

The Eilenberg--MacLane space $K(\ZZ/n, 2)$ appears as the first non-contractible term in the Postnikov tower of $\BPU_n$. Knowing the cohomology of $K(\ZZ/n,
2)$, and contrasting it with the cohomology of $\BPU_n$ which was calculated in the previous section, will allow us to determine the
extension $\BPU_n[4] \to K(\ZZ/n, 2)$. Describing what we need of $\Hoh^*(K(\ZZ/n, 2), \ZZ)$ is the purpose of this subsection. Since
$K(\ZZ/n, 2)$ is an Eilenberg--MacLane space, the cohomology classes arising may be presented as unstable cohomology operations, and we are
able to present the class that proves to be most significant subsequently, $Q$, in terms of the Bockstein and the cup-product structure on
cohomology.

Recall that we have defined $\epsilon(n)$ to be $n$ if $n$ is odd and $2n$ if $n$ is even.

 The first few integral cohomology groups of $K(\ZZ/n, 2)$ are
  \begin{equation*}
      \Hoh^1(K(\ZZ/n,2),\ZZ)=\Hoh^2(K(\ZZ/n,2),\ZZ)=\Hoh^4(K(\ZZ/n,2),\ZZ)=0,
  \end{equation*}
  and
  \begin{align*}
    \label{eq:5}
        &\Hoh^3(K(\ZZ/n,2),\ZZ) = \ZZ/n \cdot \beta, && \Hoh^5(K(\ZZ/n, 2), \ZZ) = \ZZ/\epsilon(n) \cdot Q_n.
  \end{align*}
  where $Q_n$ denotes either $\beta(\iota^2)$ if $n$ is odd, or a class such that $2Q_n = \beta(\iota^2)$ if $n$ is even.

  This calculation of cohomology may be deduced from the calculations of homology given in \cite{cartan} 
  or directly from the Serre spectral sequence for $K(\ZZ/n,1)\rightarrow*\rightarrow K(\ZZ/n,2)$. In the case where $n$
  is even, the class $Q_n$ is $2n$-torsion and a generator of $\Hoh^5(K(\ZZ/n, 2), \ZZ)$, and therefore there is a unique class $P_2 \in
  \Hoh^4(K(\ZZ/n, 2), \ZZ/2n)$ with the property that $\beta_{2n} (P_2) = Q_n$. We call this class a \textit{Pontryagin square} of
  $\iota$, in extension of the ordinary usage of this term in the case when $n$ is a power of $2$.

  We adopt cohomology-operation notation for $Q_n$ and $P_2$, writing $Q_n(\xi)$ in place of $Q_n \circ \xi$ if $\xi \in \Hoh^2(X,\ZZ)$ for
  some $X$ and similarly for $P_2$. We shall often write $Q$ for $Q_n$ when the integer $n$ is clear from the context.

  If $n\ge 2$ is an even integer, and $\xi \in \Hoh^2(X, \ZZ/n)$ for some $X$, then $\xi^2$ is a class in $\Hoh^4(X, \ZZ/n)$, and $P_2(\xi)$
  is a class in $\Hoh^4(X, \ZZ/2n)$ with the property that $2P_2(\xi)$ agrees with the image of $\xi^2$ under the map $\Hoh^4(X, \ZZ/n) \to
  \Hoh^4(X, \ZZ/2n)$ induced by the inclusion $\ZZ/n \to \ZZ/2n$. In this case we have
  \[ Q(\xi) = \beta_{2n} (P_2(\xi)). \]
  In the case of $n$ odd, and $\xi \in \Hoh^2(X, \ZZ/n)$, we have the simpler identity
  \[ Q(\xi) = \beta_n(\xi^2). \]  

\subsection*{The identification of $f_n$.}
This subsection establishes the first stage of the Postnikov tower for $\BPU_n$. The map $f_n$ of the title is the classifying map of the
first nontrivial fibration that appears in that tower.

\begin{lemma} \label{l:PB} 
  The map $\BPU_n \to K(\ZZ, 3)$ of diagram \eqref{eq:5} can be factored as $\BPU_n \to \BPU_n[2] \simeq K(\ZZ/n, 2)$ and $\beta: K(\ZZ/n,
  2) \to K(\ZZ,3)$.
  \begin{proof}
    This follows from comparing the presentation of $\PU_n$ as the quotient of $\U_n$ by $S^1$ and its presentation as the quotient of
    $\SU_n$ by $\ZZ/n$.
  \end{proof}
\end{lemma}

The map classifying $\BPU_n[4] \to K(\ZZ/n ,2)$ is a map 
\begin{equation} \label{eq:whatisfn} f_n: \BPU_n[3] = K(\ZZ/n ,2 ) \to K(\ZZ, 5) = K( \pi_4( \BPU_n), 5) ,\end{equation}
which represents a cohomology class
\[ f_n \in \Hoh^5(K(\ZZ/n, 2), \ZZ).\]


\begin{lemma} \label{l:mult}
  The $r$-fold direct-summation homomorphism $\PU_n \to \Pc(n,rn)$ followed by the reduction $\Pc(n,rn) \to \PU_{rn}$ induces a map of
  Postnikov towers for $\B \PU_n \to \B \PU_{rn}$. The functorial map 
  \[ K(\ZZ/n ,2) =  \B\PU_n[3] \to \B\PU_{rn}[3] = K(\ZZ/(rn) , 2) \]
  induces
  \[ \Hoh^5(K(\ZZ/(rn) ,2), \ZZ) \to \Hoh^5(K(\ZZ/n,2), \ZZ) \]
  taking $f_{rn}$ to $r\cdot f_n$.
\end{lemma}
\begin{proof}
  The $r$-fold direct summation map, followed by the reduction, is a homomorphism $h: \PU_n \to \PU_{rn}$ which induces a map $\B h: \BPU_n \to \BPU_{rn}$
  \begin{align*}
    \pi_2(\BPU_n) \to \pi_2(\BPU_{rn}) & \quad \ZZ/n \hookrightarrow \ZZ/rn  \quad \text{as a standard inclusion,}\\
    \pi_4(\BPU_n) \to \pi_4(\BPU_{rn} & \quad \ZZ \hookrightarrow \ZZ \quad \text{as the map $1 \mapsto r$,}
  \end{align*}
  both consequences of Lemma \ref{lem:homotopymaps}.

  The map $\B h$ gives rise to a functorial map of Postnikov towers, including 
  \[ 
  \xymatrix@C=-15pt{ \BPU_n[4] \ar^{\B h[4]}[dr]  \ar[dd] \\ & \BPU_{rn}[4] \ar[dd] \\ K(\ZZ/n ,2) = \BPU_n[3] \ar^{\B h[3]}[dr] \ar'[r]
    [rr] ^(0.2){f_n} & & K(\pi_4(\BPU_n), 5) = K(\ZZ,5)
    \ar^{\times r}[dr] \\ &
  K(\ZZ/nr, 2) =  \BPU_{rn}[3] \ar^{f_{rn}}[rr] & & K(\pi_4(\BPU_{rn}), 5) = K(\ZZ,5) }
  \]
  where the horizontal maps are classifying maps for the vertical maps, which are $K(\ZZ, 4)$--bundle maps. The map indicated by $\times r$ is
  the map obtained by applying the functor $K( \cdot, 5)$ to $\pi_4(\BPU_n) \to \pi_4(\BPU_{rn})$, which is an inclusion of $\ZZ$ in $\ZZ$ given by
  $1 \mapsto r$, and consequently the map in the diagram represents the operation $\times r$ on $\Hoh^5( \cdot, \ZZ)$. 
  Because of the commutativity of the bottom square, the map $\B h[3]^* : \Hoh^5(K(\ZZ/nr, 2), \ZZ) \to \Hoh^5(K(\ZZ/n, 2), \ZZ)$ takes
  $f_{nr}$ to $r \cdot f_n$, which is the assertion of the lemma.
\end{proof}

\begin{proposition}\label{p:generator}
  Let $n \ge 2$ be an integer. The element $f_n \in \Hoh^5(K(\ZZ/n,2), \ZZ)$ is a generator.
\end{proposition}
\begin{proof}

   Associated to the fibration $\BPU_n[4] \to K(\ZZ/n, 2)$, we have a Serre spectral sequence for integral cohomology, part of the $\Eoh_2$-page of which is
  illustrated in Figure \ref{fig:sss}.
  \begin{figure}[h]
    \centering
    \begin{equation*}
      \xymatrix@R=2px@C=10px{ \ZZ \ar^{f_n}[ddddrrrrr] \\ 0 \\ 0 \\ 0 \\ \ZZ & 0 & 0 & \ZZ/n & 0 &
        \ZZ/(\epsilon(n)) }
    \end{equation*}
    \caption{The Serre spectral sequence associated to $\BPU_n[4] \to K(\ZZ/n,2)$}
    \label{fig:sss}
  \end{figure}
  Therefore $\Hoh^5(\BPU_n[4], \ZZ) = \dfrac{\ZZ}{\epsilon(n)}\Big/\im(f_n)$.

  In the case where $n>2$, because $\pi_5(\BPU_n) = 0$, it follows that $\BPU_n[4] \simeq
  \BPU_n[5]$, and so $\Hoh^5(\BPU_n[4], \ZZ) = \Hoh^5(\BPU_n, \ZZ) = 0$, and $f_n$ is consequently surjective.

  In the case where $n=2$, we discover another extension problem, associated to the
  principal fibration $K( \ZZ/2 , 5) \to \BPU_2[5] \to \BPU_2[4]$.
  \begin{figure}[h]
    \centering
    \begin{equation*}
      \xymatrix@R=2px@C=10px{ \ZZ/2 \ar^{d_6}[dddddrrrrrr] \\ 0\\ 0 \\ 0 \\ 0 \\ \ZZ & 0 & 0 & \ZZ/2 & 0 &
        \ZZ/4 \smash{\Big/} \im f_2 & \ast }
    \end{equation*}
    \caption{The Serre spectral sequence associated to $\BPU_2[5] \to \BPU_2[4]$}
    \label{fig:sss2}
  \end{figure}
  Since $\Hoh^5(\BPU_2[5], \ZZ) = \Hoh^5(\BPU_2, \ZZ) = 0$, there is an exact sequence
  $0 \to \ZZ/4\Big/ \im f_2 \to 0 \to \ZZ/2 \to \ast$. By examining the part of the Serre
  spectral sequence appearing in Figure~\ref{fig:sss2}, it follows that $f_2$ is surjective as well.
\end{proof}

We now know that $f_n$ and $Q$ are both generators of the group $\Hoh^5(K(\ZZ/n, 2), \ZZ)$. Although $Q(\xi)$ and $f_n \circ \xi$ may not
coincide, for a given $\xi \in \Hoh^2(X, \ZZ/n)$, we know that $\langle f_n \circ \xi \rangle = \langle Q(\xi) \rangle$, and in particular
that $\ord(f_n \circ \xi) = \ord(Q_n(\xi))$.

\section{Relating the Obstruction Theory to the Atiyah--Hirzebruch Spectral Sequence} \label{sec:relat-obstr-theory}

This section is a synthesis of the work of the preceding three. The first order of business,
in Proposition \ref{p:loweasy}, is to show that
the operation $Q$ of the previous section can be used to solve the period--index problem on low-dimensional CW complexes.
In Theorem \ref{thm:relations} we show that the operation $G$ of Definition \ref{d:wasA} and the operation $Q$, which arose in the study of the
low-dimensional part of the Postnikov tower of $\BPU_n$ in the preceding two sections, generate the same subgroup of $\Hoh^5(X, \ZZ) /
(\Hoh^2(X, \ZZ) \smile \alpha)$, and therefore furnish the same obstructions. The main result of the section is Theorem A, which sets out
how the cohomology in dimensions $i \le 5$ of a finite connected CW complex $X$ may be used to give bounds on the topological index of a given
class $\alpha \in \Br'_\topo(X)$. Theorem B then establishes a sharp upper bound on the topological index in terms of the topological period
on finite CW complexes of dimension no greater than $6$, disproving the straw-man conjecture of the Introduction.

\begin{proposition} \label{p:loweasy}
  Let $X$ be a finite CW complex of dimension $d\le 5$ such that $\Hoh^2(X, \ZZ) = 0$. Let $\alpha \in \Hoh^3(X, \ZZ)$ be a
  class of order $n$, and let $\xi$ be the unique class in $\Hoh^2(X,\ZZ/n)$ such that $\beta_n(\xi) = \alpha$, then
  \[\ind_{\topo}(\alpha) = \per_{\topo}(\alpha) \ord( Q(\xi)).\] 
\end{proposition}
\begin{proof}
  For any $r$ there exists a factorization of $\alpha$ as follows
  \begin{equation*}
     \xymatrix{ X \ar[r] & K(\ZZ/nr, 2) \ar[r] & K(\ZZ, 3),}
  \end{equation*}
  and this factorization is unique up to homotopy since any two different factorizations would differ by a class in the image of $\Hoh^2(X,
  \ZZ) \to \Hoh^2(X, \ZZ/nr)$. Since $X$ is of dimension $d \le 5$, the only obstruction to finding a lift
  \begin{equation*}
    \xymatrix{  & \BPU_{nr} \ar[d] \\ X  \ar@{-->}[ur] \ar[r] & K(\ZZ/nr, 2) }
  \end{equation*}
  is the first one arising in the Postnikov tower of $\BPU_{nr}$, as presented in diagram \eqref{eq:6} and the subsequent discussion. There
  is a lift if and only if \[X \to K(\ZZ/nr, 2) \overset{f_{nr}}\longrightarrow K(\ZZ, 5)\] is nullhomotopic, which it is if and only if
  $r(f_n \circ \xi )\simeq 0$ by Lemma \ref{l:mult}, or equivalently if and only if $r | \ord(f_n \circ \xi)$. The result follows, since $\ord(f_n \circ
  \xi) = \ord(Q(\xi))$ by the discussion following Proposition \ref{p:generator}.
\end{proof}


\begin{figure}[h]
  \centering \begin{equation*}
   \xymatrix@C=12px@R=3px{
     \ZZ \ar_{\smile \beta}[ddrrr] \ar^{d_5^\beta(n)}@/^1.1em/[ddddrrrrr] & 0 & 0 & \ZZ/n & 0 & \ZZ/\epsilon(n)\\
     0 \\
     \ZZ & 0 & 0 & \ZZ/n & 0 & \ZZ/\epsilon(n) \\
     0 \\
     \ZZ & 0 & 0 & \ZZ/n & 0 &\ZZ/\epsilon(n) \\ 
   }  
  \end{equation*}
  \caption{The Atiyah--Hirzebruch spectral sequence for $\KU^*(K(\ZZ/n, 2))_\beta$, twisted by the Bockstein.}
  \label{fig:AHSSK}
\end{figure}

Recall, if $g \in A$ is an element of an abelian group $A$, that we write $\langle g
\rangle$ for the cyclic subgroup generated by $g$.

\begin{theorem}\label{thm:relations}
  Let $X$ be a connected finite CW complex and let $n$ be an integer, and $\xi \in \Hoh^2(X, \ZZ/n)$ be a cohomology class.
  Let $G$ be the operation of Theorem~\ref{thm:wasA}. Write $\tilde Q(\xi)$ for the image of $Q(\xi)$ in
  $\Hoh^5(X,\ZZ)/(\Hoh^2(X,\ZZ)\smile\beta_n(\xi))$, then $\tilde Q(\xi)$ is $\epsilon(n)$-torsion, and 
\[\left \langle \frac{n}{\per_{\topo}(\beta_n(\xi))}G(\beta_n(\xi)) \right \rangle= \left\langle \tilde Q(\xi) \right\rangle.\] 
\end{theorem}
\begin{proof}
  We write $\beta$ for $\beta_n$ throughout.
  The statement regarding the torsion of $\tilde Q(\xi)$ is immediate, since $Q(\xi)$ is $\epsilon(n)$-torsion.

  There is a classifying map $X\xrightarrow{\xi} K(\ZZ/n, 2)$, which induces a map of spectral sequences converging conditionally to the map
  $\KU^*(K(\ZZ/n, 2))_{\beta} \to \KU^*(X)_{\beta(\xi)}$. We consider the differentials supported by $\ZZ = \Eoh_2^{0,0}$ for both spaces. The
  $d_3$-differentials take the form
  \begin{equation*}
    \xymatrix{ \ZZ \ar^{d_3^{\beta}}[d] \ar@{=}[r] & \ZZ\ar^{d_3^{\beta(\xi)}}[d] \\ \Hoh^3(K(\ZZ/n, 2), \ZZ) \ar[r] &\Hoh^3(X, \ZZ).}
  \end{equation*}

  The $d_5$-differentials take the form
  \begin{equation*}
    \xymatrix{ n \ZZ \ar@{^(->}[r] \ar^{d_5^\beta}@{->>}[d] & \per_{\topo}(\beta(\xi))\ZZ \ar^{d_5^{\beta(\xi)}}[d] \\ \ZZ/\epsilon(n) \ar[r] & \Hoh^5(X, \ZZ)/
      (\Hoh^2(X, \ZZ)\smile\beta(\xi))}
  \end{equation*}
  The left vertical map is $n\mapsto G(\beta)$, and is surjective. In particular $d^\beta_5(n) $ is a unit multiple of $\tilde Q$, since $Q$ generates
  $\Hoh^5(K(\ZZ/n, 2), \ZZ)$. The top horizontal map takes the generator $n \in n \ZZ$ to $\frac{n}{\per_{\topo}(\beta(\xi))} \per_{\topo}(\beta(\xi)) \in \per_{\topo}(\beta(\xi)) \ZZ$. The
  bottom horizontal map takes $Q$ tautologically to $\tilde Q(\xi)$. The right vertical map therefore takes $\frac{n}{\per_{\topo}(\beta(\xi))} \per_{\topo}(\beta(\xi))$ to
  $u\tilde Q(\xi)$ for some unit $u$. The class of $d_5^{\beta(\xi)}(\per_{\topo}\beta(\xi))$ in $\Hoh^5(X, \ZZ)/
  (\Hoh^2(X, \ZZ)\smile\beta(\xi))$ is $G(\beta(\xi))$ by
  definition, so the relation 
  \[ \left\langle\frac{n}{\per_{\topo}(\beta(\xi))} G(\beta(\xi))\right\rangle =
      \left\langle \tilde  Q(\xi)\right\rangle \]
  follows.
\end{proof}

If $\alpha \in \Hoh^3(X, \ZZ)_{\tors}$ is any class, then there exists some lift $\xi \in \Hoh^2(X, \ZZ/n)$ such that $\beta(\xi) = \alpha$
although this lift need not be unique. By this method, the order of $G(\alpha) = G(\beta(\xi))$ may be computed, and one deduces further
that $G(\alpha)$ is $\epsilon(n)$-torsion, since $\xi$ is $n$-torsion.

Whereas the previous theorem assumed only that the class $\xi \in \Hoh^2(X, \ZZ/n)$, and therefore the class $\beta_n(\xi) = \alpha$ was
$n$-torsion, the following is more restrictive in that it assumes that the order of the class in question is precisely $n$. This simplifies
the result at the expense of little generality.

\begin{namedtheorem}[A]
  Let $X$ be a connected finite CW complex, let $\alpha \in \Hoh^3(X,\ZZ)_\tors = \Brt(X)$
  be a Brauer class, and write $n = \pert(\alpha)$.
  Choose $\xi \in \Hoh^2(X,\ZZ/n)$ such that $\beta_n(\xi) = \alpha$, let $Q(\xi) $ denote either $\beta_n(\xi^2)$ if $n$ is odd or
  $\beta_{2n}(P_2(\xi))$ if $n$ is even, and let $\tilde Q(\xi)$ denote the reduction of $Q(\xi)$ to $\Hoh^5(X, \ZZ)/ ( \alpha \smile
  \Hoh^2(X, \ZZ))$. Then:
  \begin{enumerate}
  \item The class $\tilde Q (\xi)$ depends only on $\alpha$. We write $\tilde Q(\alpha)$ for
      this class.
  \item The order of $\tilde Q(\xi)$ divides $\pert(\alpha)$ if $\pert(\alpha)$ is odd and divides $2 \pert(\alpha)$ if $\pert(\alpha)$
    is even.
  \item For any $\alpha$ and $\xi$,
    \begin{equation*}
      \pert(\alpha) \ord(\tilde Q(\xi)) | \indt(\alpha),
    \end{equation*}
    with equality if the dimension of $X$ is at most $6$.
  \end{enumerate}    
\begin{proof}
  We first show that $\tilde Q(\xi)$ depends only on $\alpha$. Any two choices of class $\xi, \xi' \in \Hoh^2(X, \ZZ/n)$ with the property
    that $\beta(\xi) = \beta(\xi') = \alpha$ differ by a class $\xi - \xi' = \nu$ where $\nu $ is in the image of the reduction map
    $\Hoh^2(X, \ZZ) \to \Hoh^2(X, \ZZ/n)$. The universal example of a space equipped with two classes $\xi, \xi'$ for which $\beta(\xi) =
    \beta(\xi')$ is the space $Y = K(\ZZ/n, 2) \times K(\ZZ, 2)$, which is equipped with two
    canonical classes $\tau \in \Hoh^2(Y, \ZZ/n)$ and $\tau+\sigma$, where
    $\sigma \in \Hoh^2(Y, \ZZ)\iso\ZZ$ is a generator. There is a map $\phi: K(\ZZ/n, 2) \to Y$ splitting the projection map. We write $\rho$ for the
    generator of $\Hoh^2(K(\ZZ/n, 2), \ZZ/n)$ and $\bar \sigma$ for the reduction of the class $\sigma$ to $\ZZ/n$ coefficients. The
    universality of $Y$ means that there is a map $f: X \to Y$ such that $f^*(\tau) = \xi$ and $f^*( \bar \sigma) = \nu$. Since
    $\tilde Q(\cdot)$ is natural, it suffices to show that $\tilde Q(\tau)$ and $\tilde
    Q(\tau + \bar \sigma)$ agree in $\Hoh^5(Y, \ZZ)/
    (\beta(\tau) \smile \Hoh^2(Y, \ZZ))$.

    The K\"unneth formula implies that $\Hoh^2(Y , \ZZ)\iso\Hoh^2(K(\ZZ, 2),\ZZ)\iso\ZZ$ and, because of the lack of torsion in $\Hoh^*(K(\ZZ, 2), \ZZ)$, that
    \[\Hoh^5(Y , \ZZ) \iso \Hoh^5( K(\ZZ/n, 2), \ZZ) \oplus \left(\Hoh^3(K( \ZZ/n, 2), \ZZ) \tensor \Hoh^2(K(\ZZ, 2) , \ZZ) \right). \]
    The splitting map, $\phi$, induces
    \begin{align*}
      & \phi^*: \Hoh^2( Y , \ZZ/n) \to \Hoh^2(K(\ZZ/n, 2) , \ZZ/n) \text{\qquad satisfying $\phi^*(\tau) = \rho$ and $\phi^*(\bar \sigma)
        =0$,}\\
      & \phi^*: \frac{ \Hoh^5(Y, \ZZ/n) }{ \beta(\tau) \smile \Hoh^2(Y, \ZZ) } \iso \frac{\Hoh^5(K(\ZZ/n, 2), \ZZ) }{ \beta(\tau) \smile
        \Hoh^2(K(\ZZ/n, 2), \ZZ) }\iso\Hoh^5(K(\ZZ/n,2),\ZZ). 
    \end{align*}
    By naturality, $\tilde Q ( \tau + \bar \sigma)  = (\phi^*)^{-1} \tilde Q (\rho) = \tilde Q( \tau) $, as required.

    The statements regarding the order of $\tilde Q( \xi)$ and the divisibility relation follow immediately from Theorem \ref{thm:relations}.
  \end{proof}
\end{namedtheorem}

  All subsequent appeals to $G(\alpha)$ will reduce to consideration of $\ord(G(\alpha))$, which can be calculated by lifting $\alpha$ to
  $\xi$ and calculating $\ord(\tilde Q(\xi))$.


\begin{namedtheorem}[B]
    Let $n$ be a positive integer, and let $\epsilon(n)$ denote $\gcd(2,n)n$. There exists a connected finite CW complex $X$ of
    dimension $6$ equipped with a class $\alpha \in \Brt(X)$ for which $\pert(\alpha) = n$ and $\indt(\alpha)=\epsilon(n)n$.
\end{namedtheorem}
\begin{proof}
    It suffices to take $X = \sk_6(K(\ZZ/n, 2))$, for which $\Hoh^2(X, \ZZ) = 0$, and to take $\alpha = \beta_n(\iota) \in \Hoh^3(X, \ZZ)$,
    a generating $n$-torsion class. Since $\Hoh^i(K(\ZZ/n),2)\rightarrow \Hoh^i(X, \ZZ)$ is
    an injection in the range $0 \le i \le 5$, the
    computation of $\Hoh^*(K(\ZZ/n, 2), \ZZ)$ implies that $\tilde Q(\beta_n(\iota))$  and $G(\alpha)$ have order $\epsilon(n)$. The result is now
    immediate from Theorem A.
\end{proof}

\section{Smooth Complex Varieties}\label{sec:3folds}

The first part of this paper culminates with the proof of Theorems A \& B. The final two sections are devoted to applications of these
results to algebraic problems, first to the algebraic period--index problem and second to the question of representing period-$2$
classes by Clifford algebras.

We turn now to the specific case of smooth complex varieties. We are chiefly interested in the case of smooth projective complex $3$--folds;
any such variety has the homotopy type of a finite $6$--dimensional CW complex and so the topological methods in the preceding sections give
a full description of the topological period--index problem on it. The topological bound of Theorem A can be used to give a lower bound on the classical
index of a class $\alpha \in \Br(X)$ in the classical Brauer group, and the nature of this bound suggests it may be possible to find a
counterexample to the period--index conjecture for function fields at the prime $2$ in dimension $3$. We have not been able to construct
such an example, so we console ourselves by constructing a higher-dimensional variety that exhibits the cohomological behavior that would
disprove the conjecture if it were seen in a $3$--fold. We do not believe that our failure to disprove the period--index conjecture in this
case is evidence in its favor, since few calculations of the integral cohomology of complex $3$--folds appear in the literature, and the
argument we might make would rely on such calculations.

If $X$ is a smooth complex variety and $\alpha\in\Br(X)$, we let $\bar{\alpha}$ denote the image of $\alpha$ under
the natural surjective map $\Br(X)\rightarrow\Br_{\topo}(X)=\Hoh^3(X,\ZZ)_{\tors}$.




Suppose we wanted to disprove the period--index conjecture. We might essay to find a smooth complex $3$-fold $X$ and a
class $\bar{\alpha}\in\Hoh^3(X,\ZZ)_{\tors}$ such that $\per_{\topo}(\bar{\alpha})=2$ and $\ind_{\topo}(\bar{\alpha})=8$.
Then we could lift it to a class $\alpha\in\Br(X)$ such that $\per(\alpha)=2$ by comparison of \'etale and singular cohomology. The
divisibility relation
\begin{equation*}
    \ind_{\topo}(\bar{\alpha})|\ind(\alpha).
\end{equation*}
is easily shown, and is treated of in \cite{aw1} at any rate.

Write $\Spec K \rightarrow X$ for the generic point of the variety. It is known that $\Br(X)\rightarrow\Br(\Spec K)$ is injective, and so
$\per(\alpha_{K})=\per(\alpha)=2$, where $\alpha_K$ is the image of $\alpha$ under the map.

\begin{proposition}\label{p:Saltman}
    Let $X$ be a regular noetherian scheme with function field $K$, and let $\alpha\in\Br(X)$. Then,
    $\ind(\alpha)=\ind(\alpha_K)$.
\end{proposition}
\begin{proof}[Proof due to D.~Saltman]
      In general it holds that
        \[
            \ind(\alpha_K)|\ind(\alpha),
        \]
        and it suffices to prove the reverse divisibility relation.

        For a given Brauer class $\alpha$, there is a notion of $\alpha$-twisted coherent sheaf, see~\cite{lieblich} for particulars. These
        objects and their morphisms form an abelian category, and we may take the $K$-theory of this category, obtaining a group
        $\G_0^{\alpha}(X)$. Similarly, there is a notion of $\alpha$-twisted vector bundle, and we can produce $K$-groups
        $\K_0^{\alpha}(X)$. The index of $\alpha$ may be described as a generator of the image of the natural rank map
        $\K_0^{\alpha}(X)\rightarrow\ZZ$. Any $\alpha_K$-twisted vector bundle over $\Spec K$ with rank $\ind(\alpha_K)$ can be extended,
        via extension of coherent sheaves, to an $\alpha$-twisted coherent sheaf on $X$ with rank $\ind(\alpha_K)$. Since $X$ is regular,
        $\G_0^{\alpha}(X)\iso\K_0^{\alpha}(X)$, so we see that
        \[
            \ind(\alpha)|\ind(\alpha_K).
        \]
        The result follows.
    \end{proof}

In summary, from Theorem A and Proposition \ref{p:Saltman}, we deduce

\begin{namedtheorem}[C]\label{thmc}
  If there exists a smooth projective complex $3$-fold $X$ and a $2$-torsion class $\alpha\in\Br(X)$ such that $2G(\bar{\alpha})\neq
  0$, where $\bar{\alpha}$ is the image of $\alpha$ in $\Hoh^3(X,\ZZ)_{\tors}$, then the period--index conjecture fails at the prime
  $2$ for the function field $\CC(X)$ of $X$: the image of the class $\alpha$ in $\Br(\CC(X))$ has period $2$ and index at least $8$.
\end{namedtheorem}

One may compare this theorem to an unpublished result of Daniel Krashen saying that if $X$ is a complex $3$-fold and $\alpha\in\Br(\CC(X))$ has period $2$,
then $\ind(\alpha)$ divides $8$.

Although it is beyond us at present to construct an example of a smooth projective complex $3$--fold of the kind called for in Theorem C, we
are able to construct the following.

\begin{proposition} \label{p:YXcons} Let $n$ be a positive integer, and $r$ a divisor of $\epsilon(n)$. Then there exists a connected smooth affine
  variety $X_{n,rn}$ of dimension $5$ such that $\Hoh^i(X_{n,rn}, \ZZ) = 0$ for $i=1,2$, and a class $\alpha_X \in \Brt(X_{n,rn})$ such that
  there is a natural isomorphism
  \[\Br(X_{n,rn}) \iso \Brt(X_{n,rn}) \iso \ZZ/n \cdot \alpha_X \]
  and such that $\ord(G(\alpha_X)) = r$. 
\end{proposition}

In particular, we can find a smooth affine $5$--fold $X_{2,8}$ on which there exists a Brauer class $\alpha$ of topological period $2$ with $\ord(G(\alpha))
=4$, and since affine $n$--folds have the homotopy type of a CW complex of dimension $n$, Theorem A asserts the index is $\indt(\alpha) =
8$.

\begin{lemma} \label{l:BrIso}
  If $X$ is a smooth complex variety, then the natural map $\Br(X) \to \Brt(X)$ is surjective. If the map $\Pic(X) \to \Hoh^2(X, \ZZ)$ is
  surjective as well, then $\Br(X) \iso \Brt(X)$.
\end{lemma}
\begin{proof}
  The hypotheses on $X$ ensure that $\Br(X) = \Br'(X)$ and $\Brt(X) = \Brt'(X)$.  For any integer $n$ there exists a natural map of long exact sequences
  \begin{equation*}
    \xymatrix{ \Pic(X) \ar[r] \ar[d] & \Hoh_\et^2(X, \mu_n) \ar@{=}[d] \ar[r] & \Br(X) \ar^{\times n}[r] \ar[d] & \Br(X) \ar[d] \\
      \Hoh^2(X, \ZZ) \ar[r] & \Hoh^2(X, \ZZ/n) \ar[r] & \Brt(X) \ar^{\times n}[r] & \Brt(X) }
  \end{equation*}
  If $\alpha \in \Brt(X)$, then $\alpha$ has finite order, so there is some $n$ for which $\alpha$ is in the image of the map $\Hoh^2(X,
  \ZZ/n)$. By a diagram-chase, it follows that $\alpha$ is in the image of the map $\Br(X) \to \Brt(X)$.

  If $\gamma \in \Br(X)$, and $\gamma$ has order $n \neq 1$, then we may lift $\gamma$ to a class in $\Hoh^2_\et(X, \mu_n)$; a diagram-chase now
  shows that if $\Pic(X) \to \Hoh^2(X, \ZZ)$ is surjective, the image of $\gamma$ in $\Brt(X)$ is nonzero.
\end{proof}

\begin{proof}[Proof of Proposition \ref{p:YXcons}]
  The natural $r$-fold summation map $\BPU_n=\B \Pc(n,n) \to \B \Pc(n, rn)$ gives rise to the following square in the
  low-dimensional part of the Postnikov towers, by reference to the results of Section \ref{sec:obstrbpu},
  \begin{equation*}
    \xymatrix{ K(\ZZ/n, 2) \ar[r] \ar@{=}[d] &  K( \ZZ, 5) \ar^{1 \mapsto r}[d] \\ K(\ZZ/n, 2) \ar^{f_n}[r] & K(\ZZ, 5) }
  \end{equation*}
  from which it follows that the extension $K(\ZZ,5) \to  \B\Pc(n,rn)[4] \to \B\Pc(n,rn)[3]$ is classified by $r$ times a generator of
  $\Hoh^5(K(\ZZ/n, 2), \ZZ)$. We deduce, from a Serre spectral sequence for instance, that $\Hoh^i(\B \Pc(n,rn), \ZZ) = 0$ for $i=1,2$, that
  \begin{equation*} \begin{split}
      \Hoh^3(\B \Pc(n,rn), \ZZ)  \iso \Hoh^2(\B \Pc(n,rn), \ZZ/n) = \Hoh^2(\B \Pc(n,rn)[3], \ZZ/n) = \Hoh^2(K(\ZZ/n, 2), \ZZ/n) = \ZZ/n \cdot
      \xi \end{split}
  \end{equation*}
  and also that
  \begin{equation*}
    \Hoh^5(\B \Pc(n,rn) ,\ZZ)  = \Hoh^5( \B \Pc(n,rn)[4], \ZZ) =\ZZ/r \cdot Q(\xi) 
  \end{equation*}
  where $Q(\xi)$ is used in the same sense as in the statements of Theorems \ref{thm:relations} and B. Defining $\alpha = \beta_n(\xi)$, and
  observing that $\Hoh^2(\B \Pc(n,rn), \ZZ) = 0$, we employ Theorem A to deduce that $\ord(G(\alpha)) = r$.

  The space $\B \Pc(n,rn)$ is not a complex variety, although it is homotopy equivalent to
  $\B \SL_{rn}(\CC)/\mu_n$, which is the classifying space of
  an algebraic group. There are a number of results regarding the approximation of
  classifying spaces of complex algebraic groups by complex varieties; following
  \cite{totaro} we may construct a quasi-projective variety which approximates $\B
  \SL_{rn}(\CC)/\mu_n$ in that there is a map $X_{n,rn} \to \B\SL_{rn}(\CC)/\mu_n$
  that induces an isomorphism on cohomology groups $\Hoh^{\le 4}(\cdot, \ZZ)$ and an injection on $\Hoh^5(\cdot, \ZZ)$ . By use of
  Jouanolou's device~\cite{jouanolou}, $X_{n,rn}$ can be assumed to be affine, and by use of the affine Lefschetz hyperplane theorem, \cite{goresky-macpherson}, it can be
  assumed to have dimension $5$.
\end{proof}

The spaces $\B\SL_{rn}(\CC)/\mu_n$ are like universal examples of spaces with period-$n$, degree-$rn$ projective bundles, so it
should not be a surprise that we might find classes of period $n$ and index $rn$ in the Brauer groups of $X_{n,rn}$.

The same kind of argument can be used to construct a smooth projective $5$-fold with a Brauer class
of period $n$ and topological index $rn$. Starting from the algebraic group
$G=\SL_{rn}/\mu_n \times \Gm$,
which has positive-dimensional center, we may use a result of Ekedahl from \cite{ekedahl}
constructing a smooth projective variety $Y_{n,rn}$ the cohomology of which agrees with that of $\B \Pc(n,rn) \times \CC P^\infty$ in the range
$\Hoh^{\le 4}(\cdot, \ZZ)$ and for which there is an injection on $\Hoh^5(\cdot, \ZZ)$. The variety $Y_{n,rn}$ may be constructed so that
$\Pic(Y_{n,rn}) = \Hoh^2(Y_{n,rn}) \iso \ZZ$, and by using the Lefschetz hyperplane theorem and a Lefschetz-type theorem for the Picard
group, we may assume $Y_{n,rn}$ is $5$-dimensional.

\section{A Question of Hahn on Quadratic Forms}\label{sec:hahn}

Let $E$ be an $m$-dimensional vector space over a field $k$ of characteristic unequal to $2$, and let $q$ be a quadratic form on $E$. The
Clifford algebra $\Cl(q)$ is defined to be the quotient of the tensor algebra
\begin{equation*}
    T(E)=k\oplus E\oplus (E\otimes_k E)\oplus\cdots
\end{equation*}
by the ideal generated by 
\begin{equation}
  \label{eq:7}
  x\otimes x+q(x)\cdot 1
\end{equation}
for all $x\in E$.  We will need only the rudiments of the theory of Clifford algebras
over the complex field and we refer the reader to \cite{lawson-michelsohn-book} for a fuller treatment of the topic than we provide.
The algebra $\Cl(q)$ is associative and has dimension $2^m$, and the original vector space $E \subset \Cl(q)$ is a generating subset.

If $X$ is a connected topological space, then a \textit{complex quadratic bundle} on $X$ is a complex vector bundle $\mathcal{E}$ on $X$ and
a map $b: \mathcal{E} \to \mathcal{E}^*$ of vector bundles. The map on stalks $b_x: \mathcal{E}_x \to \mathcal{E}_x^*$ is used to define a bilinear form
$b_x(\alpha ,\beta ) \in\CC$ on the vector space $\mathcal{E}_x$, and from that one defines the quadratic form $q(\alpha ) = b_x(\alpha,
\alpha)$. Since $\CC$ is not of characteristic $2$, the theory of bilinear and quadratic forms agree. A quadratic bundle $(\mathcal{E}, b)$
is said to be \textit{nondegenerate} if $b: \mathcal{E} \to \mathcal{E}^*$ is injective, and similarly a quadratic form on a vector space is
nondegenerate if the associated bilinear form $b: V \to V^*$ is injective. In the case of finite rank or dimension, injectivity is
synonymous with isomorphism.

Every nondegenerate quadratic form on the $m$-dimensional vector space $E \iso \CC^m$ is isomorphic to the standard form:
\begin{equation*}
    q(x)=x_1^2+\cdots +x_m^2.
\end{equation*}
We denote the Clifford algebra of the standard form by $\CCl_m$. There are isomorphisms
\begin{align*}
    \CCl_{2n} &= \Mat_{2^n \times 2^n}(\CC)  \qquad  \text{if $m=2n$,} \hfill \\
    \CCl_{2n+1} &= \Mat_{2^n \times 2^n}(\CC) \oplus \Mat_{2^n \times 2^n}(\CC)  \quad \text{if $m=2n+1$.}
\end{align*}
We consider only the case of even $m=2n$ in the sequel.

There is a subgroup of $\CCl_{2n}^\times$, denoted by $\Pin_{2n}(\CC)$, generated by elements $e \in E$ for which $q(e) = \pm 1$, and a
connected subgroup of $\Pin_{2n}(\CC)$ denoted by $\Spin_{2n}(\CC)$, generated by the products of even numbers of such $e \in E$. When $n\ge
2$, the group $\Spin_{2n}(\CC)$ is simply-connected. In the case at hand, $\CCl_{2n}^\times$ is the group $\GL_{2^n}(\CC)$ and so
$\Spin_{2n}(\CC)$ is a subgroup of $\SL_{2^n}(\CC)$. The group $\SL_{2^n}(\CC)$ acts by conjugation as algebra homomorphisms on $\CCl_{2n} =
\Mat_{2^n \times 2^n}(\CC)$; the kernel of this action is the center, $\mu_{2^n}$, and the image is the group, $\PGL_{2^n}(\CC)$, of all
algebra homomorphisms of $\Mat_{2^n \times 2^n}(\CC)$.

Using the defining relation of \eqref{eq:7}, one can show that if $e_1, e_2 \in E$ are elements satisfying $q(e_1)^2 = q(e_2)^2 = 1$,
and if $v \in E$ is some third element, viewed as an element of $\CCl(E)$, that $e_1 e_2 v e_2^{-1} e_1^{-1}$ is again an element of $E
\subset \CCl(E)$, being the element given by reflection of $v$ in the hyperplanes $e_2^{\perp}$ and $e_1^{\perp}$. This means that the
conjugation action of $\Spin_{2n}(\CC)$ on $\CCl_{2n}$ fixes $E$, and since it acts on $E$ by composites of pairs of reflections, this describes a
homomorphism $\Spin_{2n}(\CC) \to \SO_{2n}(\CC)$. This homomorphism is the universal covering map of the group $\SO_{2n}(\CC)$, and has kernel $\{ \pm
I_{2^n} \} = \mu_2$.

We have constructed a map of short exact sequences of groups
\begin{equation}
  \label{eq:8}
  \xymatrix{ 1 \ar[r] & \mu_{2^n} \ar[r] & \SL_{2^n}(\CC) \ar[r] & \PGL_{2^n}(\CC) \ar[r] & 1\\
    1 \ar[r] & \mu_2 \ar[u]  \ar[r] & \Spin_{2n}(\CC) \ar[u] \ar[r] & \SO_{2n}(\CC) \ar[r] \ar[u] & 1.}
\end{equation}
There is an induced map $\Cl: \B \SO_{2n}(\CC) \to \B \PGL_{2^n}(\CC)$ which has the property that it induces the nontrivial map $\ZZ/2 \to
\ZZ/(2^n)$ on the homotopy groups $\pi_2(\B\SO_{2n}(\CC)) \to \pi_2(\B\PGL_{2^n})$. This last fact holds when $n\ge 2$ since
$\Spin_{2n}(\CC)$ and $\SL_{2n}(\CC)$ are simply-connected, and holds for trivial reasons when $n=1$. The space $\B \SO_{2n}(\CC)$
represents rank-$2n$ vector bundles that are equipped with nondegenerate quadratic forms, and for which a certain invariant, the
discriminant, vanishes; the last condition being the condition that permits the lifting of the structure group of a quadratic vector bundle
from $\mathrm{O}_{2n}(\CC)$ to $\SO_{2n}(\CC)$.

The map $\Cl$ represents a natural construction that associates to the data $(\mathcal{E}, q)$ the
principal $\PGL_{2^n}(\CC)$-bundle associated to the bundle $\Cl(\mathcal{E}, q)$ of matrix algebras. This is called the \textit{Clifford
  map}. The class $[\Cl(\mathcal{E},q)] \in \Brt(X)$ is $2$-torsion, since it is the image under the classifying map of the nontrivial
element in $\Brt(\B \SO_{2n}(\CC)) = \pi_2(\B \SO_{2n}(\CC)) = \ZZ/2$. The class does not change if one adds a trivial $\CC^2$-bundle with trivial quadratic form to
$(\mathcal{E}, q)$, and so the Clifford map extends to a map to a map, the \textit{Clifford invariant map}, $\Cl: I_\topo^2(X) \to {}_2 \Brt(X)$, where the notation $I_\topo^2(X)$ denotes the subgroup of the
Witt group of $X$ generated by quadratic bundles $(\mathcal{E},q)$ having even rank and trivial discriminant.

The Clifford invariant map is the analogue in the setting of complex quadratic vector bundles on topological spaces of a Clifford invariant
map $I^2(X) \to \Br(X)$ for schemes, which generalizes a norm residue map $I^2_F \to \Br(F)$ for fields. The norm residue map for fields was
shown to be surjective by Merkurjev in \cite{merkurjev}, a result now subsumed in the
Norm--Residue Isomorphism Theorem proved by Rost and Voevodsky. If $X$ is a complex variety,
because the Brauer class is formed in both cases as the class associated to a Clifford algebra of a quadratic bundle, there is a comparison
diagram:
\[ \xymatrix{ I^2(X) \ar[r] \ar[d] & {}_2\Br(X) \ar[d] \\ I^2_\topo(X) \ar[r] & {}_2\Brt(X). } \] In the remainder of this section, we prove
that the image of the map $I_\topo^2(X) \to {}_2 \Brt(X)$ cannot contain any classes $\alpha$ for which $\ord(G(\alpha)) =4$. We do this by
proving that $\ord(G(\alpha)) = 2$ where $\alpha$ is the class of the universal stable $\SO$ bundle. By comparison, this imposes a
restriction on the classes which may be in the image of the algebraic Clifford invariant map $I^2(X) \to \Br(X)$. We are then able to
exhibit an affine variety $X$ for which $\Pic(X) = 0$, but for which the Clifford invariant map is not surjective.

\medskip

We refer the reader to the paper of Brown, \cite{brown}, where cohomology of $\BSO_n$ is calculated in full. We make use of the
calculation only in the limiting case of $\BSO$, and then only in low degrees.
\begin{proposition}
  In low degrees, the cohomology $\Hoh^*(\BSO, \ZZ)$ is
  \begin{gather*}
    \Hoh^1 (\BSO, \ZZ) = \Hoh^2( \BSO, \ZZ) = 0, \quad \Hoh^3(\BSO, \ZZ) = \ZZ/2 \cdot \beta(w_2), \\
    \Hoh^4(\BSO, \ZZ) = \ZZ \cdot p_1, \quad \Hoh^5(\BSO, \ZZ/2) = \ZZ/2 \cdot \beta(w_4).
  \end{gather*}
\end{proposition}

\begin{proposition}
  There is an identification: $G(\beta(w_2))= \beta(w_4)$.
\end{proposition}
\begin{proof}
  
The Bott periodicity theorem asserts that the homotopy groups of $\BSO$ are, in low degrees, given by:
\begin{gather*}
    \pi_2(\BSO)=\ZZ/2,\hspace{2cm}\pi_4(\BSO) \iso\pi_8(\BSO) \iso\ZZ,\\\pi_0(\BSO)=\pi_1(\BSO)=\pi_3(\BSO)=\pi_5(\BSO)=\pi_6(\BSO)=\pi_7(\BSO)=0.
\end{gather*}
In particular, the first stage of the Postnikov tower for $\BSO$ takes the form of an extension
\begin{equation*}
  K(\ZZ, 4) \to \BSO[4] \to \BSO[3] = K(\ZZ/2, 2).
\end{equation*}
The cohomology of $\BSO[4]$ agrees with that of $\BSO$ in the range $H^{\le 7}(\cdot, \ZZ)$, by Bott periodicity, and consequently
$\Hoh^5(\BSO[4], \ZZ) \iso \ZZ/2 \cdot \beta(w_4)$.

On the other hand, $\Hoh^2(\BSO[3], \ZZ/2) \iso \ZZ/2 \cdot \iota $ and $\Hoh^5(\BSO[3], \ZZ) \iso \ZZ/4 \cdot \beta_4( P_2(\iota))$. By
 considering the Serre spectral sequence for the extension $\BSO[4] \to \BSO[3]$, one concludes that the map $\BSO \to \BSO[3]$ induces a map
 $\iota \mapsto w_2$ and $\beta_4(P_2(\iota)) \mapsto \beta(w_4)$, since there is no other possible source of a class in $\Hoh^2$ or
 $\Hoh^5$. By naturality, $\beta_4(P_2(w_2)) = \beta(w_4)$.

Theorem A implies that $G(\beta(w_2))$  and $\beta(w_4)$ agree up to automorphism in the group $\Hoh^5(\BSO, \ZZ) \iso \ZZ/2$, so they coincide.
\end{proof}

\begin{proposition}\label{prop:ind4}
  if $X$ is a complex variety, and $\alpha \in \Br(X)$ is in the image of the Clifford invariant map, $\Cl : I^2(X) \to {}_2 \Br(X)$, then
  $\ord(G(\bar{\alpha})) | 2$. In particular, if $X(\CC)$ has the homotopy-type of a CW complex of dimension no greater than $6$, then
  $\indt(\bar{\alpha})$ divides $4$.
\end{proposition}
\begin{proof}
  By the hypothesis there exists some $2k$-dimensional vector bundle $F$ on $X$, equipped with a nondegenerate quadratic form $q$ of trivial
  discriminant, such that the associated Clifford algebra $\Cl(F,q)$ represents the class
  $\alpha$ in $\Br(X)$.
  
  Topologically, the quadratic bundle $(F,q)$ is represented by a map $X \to \B \SO_{2k}(\CC)$.  The Lie group $\SO_{2k} = \SO_{2k}(\RR)$ is a compact form
  of $\SO_{2k}(\CC)$, and therefore there exists a map $f: X \to \BSO_{2k}$ classifying the pair $(F,q)$. Since the class $[\Cl(F,q)]$ is
  invariant under the addition of a trivial quadratic summand to $(F,q)$, we may replace $f: X \to \BSO_{2k}$ with a map $f: X \to
  \BSO$. Then the class $\bar{\alpha} = [\Cl(F,q)]$ which is the image of the bundle under the Clifford invariant map is $f^*( \beta(w_2))$.

  By naturality $G(\bar{\alpha})$ has order dividing
  $2=\ord(G(\beta(w_2)))$. If $X(\CC)$ has the homotopy type of a CW complex of dimension no greater than $6$, it
  follows from Theorem A that $\indt(\bar{\alpha}) | 4$.
\end{proof}

\newcommand{\GSO}{\mathrm{GSO}}

There exists a group $\mathrm{GSO}$, which fits in an extension of groups
\[ \xymatrix{ 1 \ar[r] & \SO \ar^\phi[r] & \mathrm{GSO} \ar[r] & \Gm \ar[r] & 1. } \] This group has the property that $\B \mathrm{GSO}$
classifies the line-bundle-valued quadratic modules for which the total Clifford invariant
map is defined; this is laid out in~\cite{auel-clifford}.
 By comparing Postnikov towers, one can see that there is a class $w'_2 \in \Hoh^2(\B \GSO, \ZZ/2)$ with the properties
that $\phi^*(w_2') = w_2$ and $\phi^*(Q(w_2')) = Q(w_2)$, where $Q$ is used in the sense of Section \ref{sec:relat-obstr-theory}. There is
an isomorphism
\[\Hoh^3(\B \GSO, \ZZ) = \ZZ/2 \cdot \beta(w_2') \iso \Hoh^3(\B \SO, \ZZ) = \ZZ/2 \cdot \beta(w_2) \]
and it is possible to deduce that $\ord(G(\beta(w_2'))) = \ord (G( \beta(w_2))) = 2$ by comparison. In summary, the result of Proposition
\ref{prop:ind4} holds not only for the Clifford invariant map, but also for the total Clifford invariant map.

\begin{namedtheorem}[D]
  There exists a $5$-dimensional smooth affine complex variety $X$ such that $\Br(X) \iso\ZZ/2$, but the total Clifford invariant map
  vanishes.
\end{namedtheorem}

\begin{proof}
    Let $X$ be the smooth affine complex $5$-fold $X_{2,8}$ of Proposition \ref{p:YXcons}, equipped with a class $\alpha$ of topological
    period $2$ and topological index $8$. This variety satisfies $\Br(X) = \ZZ/2 \cdot \alpha$ and $\Hoh^1(X, \ZZ) = \Hoh^2(X, \ZZ) =
    0$. We show first that $\Pic(X) = 0$, so that the total Clifford invariant reduces to the ordinary Clifford invariant, and then we shall
    show that the ordinary Clifford invariant map is trivial.

    For any integer $n \ge 2$, there exists an exact sequence of \'etale cohomology groups,
    \[ \xymatrix{ 0 = \Hoh_\et^1(X, \mu_n) \ar[r] & \Pic(X) \ar^{\times n}[r] & \Pic(X) \ar[r] & \Hoh_\et^2(X, \mu_n).} \]
    It follows that $\Pic(X)$ is torsion-free, and since $\Hoh_{\et}^2(X,\mu_n)=0$ if
    $n\geq 3$ is odd, $\Pic(X)$ is $n$-divisible for every such $n$. When $n=2$, we have an
    isomorphism $\Hoh^2_{\et}(X,\mu_2)=\Br(X)$, so that $\Pic(X)$ is $2$-divisible as well.
    To show $\Pic(X) =0$ it will suffice to show that it is finitely generated.

    By compactifying and then resolving singularities, we may embed $X$ as an open subvariety of a smooth projective variety $X \subset \tilde
    X$ where $\tilde X \setminus X$ is a normal-crossing divisor, and from this we may present $\Pic(X)$ as a quotient of $\Pic(\tilde
    X)$ by the finitely-generated abelian group generated by the classes of the irreducible components of $\tilde X \setminus X$. It will be enough to show that $\Pic(\tilde
    X)$ is finitely generated, and since $\tilde X$ is a projective variety, this is equivalent to proving that $\Pic^0(\tilde X) = 0$. If
    $\Pic^0(\tilde X)$ is not $0$, it is a nontrivial abelian variety and is of the form $\CC^r/ \Lambda$ where $\Lambda$ is a lattice. There is no bound on the
    primes $p$ for which nontrivial $p$-torsion appears in $\Pic^0(\tilde X)$, and consequently no bound on the primes for which nontrivial
    $p$-torsion appears in $\Pic(\tilde X)$ and $\Pic(X)$, contradicting the lack of torsion of $\Pic(X)$. It follows that $\Pic(X) = 0$,
    and so the total and ordinary Clifford invariant maps on $X$ agree.

    The class $\alpha$ satisfies $\ord(G(\bar{\alpha})) = 4$ and so Proposition \ref{prop:ind4} implies that $\alpha$ is not in the image of the
    Clifford invariant map. Because ${}_2\Br(X) = \Br(X) = \ZZ/2 \cdot \alpha$, the image of the Clifford map is $0$.
\end{proof}

Auel's question at the end of \cite{auel} is posed in greater generality than the case of complex varieties; restricted to that
case, it asks whether the total Clifford invariant map is surjective for smooth complex
varieties of dimension $d \le 3$. If we were
able to find a smooth projective $3$-fold, $X$, meeting the conditions of Theorem C,  with a class $\alpha \in \Br(X)$ such that
$\per(\alpha) =2$ and $8 |\indt(\overline{\alpha}) $, then the results of this section would imply a
negative answer to Auel's question. Indeed, by the comment after
Proposition~\ref{prop:ind4}, the topological index of
any class in the image of the total Clifford invariant map is at most $4$ on a complex $3$-fold.


\begin{bibdiv}
\begin{biblist}


\bib{aw1}{article}{
    author = {Antieau, Benjamin},
    author = {Williams, Ben},
    title = {The period--index problem for twisted topological K-theory},
    journal = {ArXiv e-prints},
    eprint = {http://arxiv.org/abs/1104.4654},
    year = {2011},
}

\bib{aw2}{article}{
    author = {Antieau, Benjamin},
    author = {Williams, Ben},
    title = {Serre-Godeaux varieties and the \'etale index},
    journal = {J. K-Theory},
    year = {2013},
    volume = {11},
    number = {2},
    pages={283--295},
}


\bib{aw4}{article}{
    author = {Antieau, Benjamin},
    author = {Williams, Ben},
    title = {Unramified division algebras do not always contain Azumaya maximal orders},
    journal = {Inv. Math.},
    eprint = {http://dx.doi.org/doi:10.1007/s00222-013-0479-7},
    year = {2013},
}

\bib{aw5}{article}{
    author = {Antieau, Benjamin},
    author = {Williams, Ben},
    title = {On the classification of principal $\PU_2$-bundles over a $6$-complex},
    journal = {ArXiv e-prints},
    eprint = {http://arxiv.org/abs/1209.2219},
    year = {2012},
}
\bib{atiyah-segal}{article}{
    author={Atiyah, Michael},
    author={Segal, Graeme},
    title={Twisted $K$-theory},
    journal={Ukr. Mat. Visn.},
    volume={1},
    date={2004},
    number={3},
    pages={287--330},
    issn={1810-3200},
    translation={
    journal={Ukr. Math. Bull.},
    volume={1},
    date={2004},
    number={3},
    pages={291--334},
    issn={1812-3309},
    },
}

\bib{atiyah-segal-cohomology}{article}{
    author={Atiyah, Michael},
    author={Segal, Graeme},
    title={Twisted $K$-theory and cohomology},
    conference={
    title={Inspired by S. S. Chern},
    },
    book={
    series={Nankai Tracts Math.},
    volume={11},
    publisher={World Sci. Publ.,
    Hackensack, NJ},
    },
    date={2006},
    pages={5--43},
}

\bib{auel-clifford}{article}{
    author = {Auel, Asher},
    title = {Clifford invariants of line-bundle valued quadratic forms},
    journal = {MPIM e-prints},
    eprint = {http://www.mpim-bonn.mpg.de/preblob/5091},
    year = {2011},
}

\bib{auel}{article}{
    author = {Auel, Asher},
    title = {Surjectivity of the total Clifford invariant and Brauer dimension},
    journal = {ArXiv e-prints},
    eprint = {http://arxiv.org/abs/1108.5728},
    year = {2011},
}

\bib{baum-browder}{article}{
	title = {The cohomology of quotients of classical groups},
	volume = {3},
	journal = {Topology},
	author = {Baum, Paul F.},
    author={Browder, William},
	year = {1965},
	pages = {305--336}
}


\bib{bott}{article}{
    author={Bott, Raoul},
    title={The space of loops on a Lie group},
    journal={Michigan Math. J.},
    volume={5},
    date={1958},
    pages={35--61},
    issn={0026-2285},
}

\bib{brown}{article}{
    author={Brown, Edgar H., Jr.},
    title={The cohomology of $B{\rm SO}_{n}$ and $B{\rm O}_{n}$ with
    integer coefficients},
    journal={Proc. Amer. Math. Soc.},
    volume={85},
    date={1982},
    number={2},
    pages={283--288},
    issn={0002-9939},
}

\bib{cartan}{article}{
    author = {Cartan, H.},
    title = {D{\'e}termination des alg{\`e}bres {$\Hoh_*(\pi,n;\ZZ)$}},
    journal = {S{\'e}minaire H. Cartan},
    volume = {7},
    number = {1},
    pages = {11-01--11-24},
    publisher = {Secr{\'e}tariat math{\'e}matique},
    address = {Paris},
    year = {1954/1955},
}

\bib{colliot}{article}{
    author={Colliot-Th{\'e}l{\`e}ne, Jean-Louis},
    title={Exposant et indice d'alg\`ebres simples centrales non ramifi\'ees},
    note={With an appendix by Ofer Gabber},
    journal={Enseign. Math. (2)},
    volume={48},
    date={2002},
    number={1-2},
    pages={127--146},
    issn={0013-8584},
}

\bib{dejong}{article}{
    author={de Jong, A. J.},
    title={The period--index problem for the Brauer group of an algebraic
    surface},
    journal={Duke Math. J.},
    volume={123},
    date={2004},
    number={1},
    pages={71--94},
    issn={0012-7094},
}

\bib{dejong-starr}{article}{
    author={Starr, Jason},
    author={de Jong, Johan},
    title={Almost proper GIT-stacks and discriminant avoidance},
    journal={Doc. Math.},
    volume={15},
    date={2010},
    pages={957--972},
    issn={1431-0635},
}

\bib{ekedahl}{article}{
  author = {{Ekedahl}, T.},
  title = {Approximating classifying spaces by smooth projective varieties},
  journal = {ArXiv e-prints},
  eprint = {http://arxiv.org/abs/0905.1538},
}

\bib{goresky-macpherson}{book}{
    author={Goresky, Mark},
    author={MacPherson, Robert},
    title={Stratified Morse theory},
    series={Ergebnisse der Mathematik und ihrer Grenzgebiete (3)},
    volume={14},
    publisher={Springer-Verlag},
    place={Berlin},
    date={1988},
    pages={xiv+272},
    isbn={3-540-17300-5},
}

\bib{grothendieck-brauer}{article}{
    author={Grothendieck, Alexander},
    title={Le groupe de Brauer. I. Alg\`ebres d'Azumaya et interpr\'etations diverses},
    conference={
    title={S\'eminaire Bourbaki, Vol.\ 9},
    },
    book={
    publisher={Soc. Math. France},
    place={Paris},
    },
    date={1995},
    pages={Exp.\ No.\ 290, 199--219},
}


\bib{jouanolou}{article}{
    author={Jouanolou, J. P.},
    title={Une suite exacte de Mayer-Vietoris en $K$-th\'eorie alg\'ebrique},
    conference={
    title={Algebraic $K$-theory, I: Higher $K$-theories (Proc. Conf.,
    Battelle Memorial Inst., Seattle, Wash., 1972)},
    },
    book={
    publisher={Springer},
    place={Berlin},
    },
    date={1973},
    pages={293--316. Lecture
    Notes in Math., Vol.
    341},
}

\bib{kameko-yagita}{article}{
    author={Kameko, Masaki},
    author={Yagita, Nobuaki},
    title={The Brown-Peterson cohomology of the classifying spaces of the
    projective unitary groups ${\rm PU}(p)$ and exceptional Lie groups},
    journal={Trans. Amer. Math. Soc.},
    volume={360},
    date={2008},
    number={5},
    pages={2265--2284},
    issn={0002-9947},
}

\bib{lawson-michelsohn-book}{book}{
    author={Lawson, H. Blaine, Jr.},
    author={Michelsohn, Marie-Louise},
    title={Spin geometry},
    series={Princeton Mathematical Series},
    volume={38},
    publisher={Princeton University Press},
    place={Princeton, NJ},
    date={1989},
    pages={xii+427},
    isbn={0-691-08542-0},
}

\bib{lieblich}{article}{
    author={Lieblich, Max},
    title={Twisted sheaves and the period--index problem},
    journal={Compos. Math.},
    volume={144},
    date={2008},
    number={1},
    pages={1--31},
    issn={0010-437X},
}

\bib{matzri}{unpublished}{
    author={Matzri, Eliyahu},
    title={Symbol length in the Brauer group of a $C_m$ field},
    year={2013},
}


\bib{merkurjev}{article}{
    author={Merkurjev, A. S.},
    title={On the norm residue symbol of degree $2$},
    journal={Dokl. Akad. Nauk},
    volume={261},
    date={1981},
    number={3},
    pages={542--547},
    issn={0002-3264},
}

\bib{milnor-stasheff}{book}{
    author={Milnor, John W.},
    author={Stasheff, James D.},
    title={Characteristic classes},
    note={Annals of Mathematics Studies, No. 76},
    publisher={Princeton University Press},
    place={Princeton, N. J.},
    date={1974},
    pages={vii+331},
}

\bib{parimala-sridharan-1}{article}{
    author={Parimala, R.},
    author={Sridharan, R.},
    title={Graded Witt ring and unramified cohomology},
    journal={$K$-Theory},
    volume={6},
    date={1992},
    number={1},
    pages={29--44},
    issn={0920-3036},
}

\bib{parimala-sridharan-2}{article}{
    author={Parimala, R.},
    author={Sridharan, R.},
    title={Nonsurjectivity of the Clifford invariant map},
    note={K. G. Ramanathan memorial issue},
    journal={Proc. Indian Acad. Sci. Math. Sci.},
    volume={104},
    date={1994},
    number={1},
    pages={49--56},
    issn={0253-4142},
}



\bib{totaro}{article}{
    author={Totaro, Burt},
    title={The Chow ring of a classifying space},
    conference={
    title={Algebraic $K$-theory},
    address={Seattle, WA},
    date={1997},
    },
    book={
    series={Proc. Sympos. Pure Math.},
    volume={67},
    publisher={Amer. Math. Soc.},
    place={Providence, RI},
    },
    date={1999},
    pages={249--281},
}

\bib{vistoli}{article}{
    author={Vistoli, Angelo},
    title={On the cohomology and the Chow ring of the classifying space of ${\rm PGL}_p$},
    journal={J. Reine Angew. Math.},
    volume={610},
    date={2007},
    pages={181--227},
    issn={0075-4102},
}

\bib{woodward}{article}{
    author={Woodward, L. M.},
    title={The classification of principal ${\rm PU}_{n}$-bundles over a $4$-complex},
    journal={J. London Math. Soc. (2)},
    volume={25},
    date={1982},
    number={3},
    pages={513--524},
    issn={0024-6107},
}

\end{biblist}
\end{bibdiv}
\end{document}